\documentclass[11pt,a4paper]{article}
\usepackage[utf8]{inputenc}
\usepackage[english]{babel}
\usepackage[T1]{fontenc}
\usepackage{amsmath}
\usepackage{caption}
\usepackage{version}
\usepackage{amsfonts}
\usepackage{amssymb}
\usepackage{units}
\usepackage{graphicx}
\usepackage{enumitem}
\usepackage[]{algorithm2e}
\usepackage[dvipsnames]{xcolor}
\usepackage{tikz}
\usepackage{bigints}
\usepackage[left=2.5cm,right=2.5cm,top=2.3cm,bottom=2.3cm]{geometry}
\usepackage{amsthm}
\usepackage{multicol}
\usepackage{subcaption}
\usepackage{float}
\usepackage{authblk}
\usepackage{todonotes}

\newcommand{\psh}[2]{\ensuremath{\langle #1 \, , #2 \rangle}\xspace}

\newcommand{\sumlim}[2]{\sum_{#1}^{#2}\xspace}
\newcommand{\itg}[4]{\ensuremath{\int_{#1}^{#2} #3 \, \mathrm{d}#4 }\xspace}

\newcommand{\R}{\mathbb R}
\newcommand{\N}{\mathbb N}
\newcommand{\Z}{\mathbb Z}
\newcommand{\C}{\mathbb C}

\newcommand{\br}{{\bf r}}

\newcommand{\cO}{{\mathcal O}}

\newcommand{\cE}{{\mathcal E}}

\begin{document}
\numberwithin{equation}{section}
\newtheorem{theo}{Theorem}[section]
\newtheorem{prop}[theo]{Proposition}
\newtheorem{note}[theo]{Remark}
\newtheorem{lem}[theo]{Lemma}
\newtheorem{cor}[theo]{Corollary}
\newtheorem{definition}[theo]{Definition}
%%%%%%%%%%%%%%%%%%%%%%%%%%%%%%%%%%%%%

\title{Variational projector augmented-wave method: theoretical analysis and preliminary numerical results}
\author[1]{X. Blanc}
\author[2]{E. Canc\`es}
\author[1]{M.-S. Dupuy}
\affil[1]{Univ. Paris Diderot, Sorbonne Paris Cité, Laboratoire Jacques-Louis Lions, UMR 7598, UPMC, CNRS, F-75205 Paris, France}
\affil[2]{CERMICS - Ecole des Ponts ParisTech, 6 \& 8 avenue Blaise Pascal, Cit\'e Descartes, 77455 Marne la Vall\'ee Cedex 2, France}

\renewcommand\Affilfont{\itshape\small}

\maketitle
\begin{abstract}
\begin{small}
In Kohn-Sham electronic structure computations, wave functions have singularities at nuclear positions. Because of these
singularities, plane-wave expansions give a poor approximation of the eigenfunctions. In conjunction with the use of
pseudo-potentials, the PAW (projector augmented-wave) method circumvents this issue by replacing the original eigenvalue problem by a new one with the same eigenvalues
but smoother eigenvectors. Here a slightly different method, called VPAW
(variational PAW), is proposed and analyzed. This new method allows for a better convergence with respect to the number of plane-waves. Some
numerical tests on an idealized case corroborate this efficiency. This work has been recently announced in \cite{blanc2017}. 
\end{small}
\end{abstract}
%%%%%%%%%%%%%%%%%%%%%%%%%%%%%%%%%%%%%%%%%%%%%%%%%%%%%%%%%%%%%%%%%%%%%%%%%%%%%%%

\section*{Introduction}

%Solving the $N$-body electronic problem is numerically impossible even for small molecular systems. Various approximations of this problem have been proposed, among which Hartree-Fock theory, post-Hartree-Fock methods and density functional theory, which describe fairly well the ground-state electronic structure of many molecules. However it can still be expensive to compute the desired properties with these methods. One of the main reasons is the Coulombian potential which gives rise to cusps \cite{kato1957eigenfunctions, hoffmann2001electron, fournais2005sharp} located at each nucleus that considerably impedes the rate of convergence of plane-wave expansion. Over the years, several strategies have been developped to tackle this problem. 

Solving the $N$-body electronic problem is numerically impossible even for small molecular systems. Various nonlinear one-body approximations of this problem have been proposed, among which Hartree-Fock theory, post-Hartree-Fock methods and density functional theory, which describe fairly well the ground-state electronic structure of many molecules. However it can still be expensive to compute the desired properties with these approximations. In solid-state physics, plane-wave methods are often the method of choice to compute the lowest eigenvalues of Kohn-Sham operators. However Coulomb potentials considerably impedes the rate of convergence of plane-wave expansion because of the cusps \cite{kato1957eigenfunctions, hoffmann2001electron, fournais2005sharp} located at each nucleus. Over the years, several strategies have been developped to tackle this problem. 

In most situations, the knowledge of the whole energy spectrum of the Kohn-Sham Hamiltonian is not needed. Only eigenvalues belonging to a small
energy range are. Indeed, it is well-known that the chemical properties mostly come from the valence electrons. So
it would be satisfactory to replace the Coulomb potential and nonlinear interactions with the core
electrons by a smooth potential that reproduces the exact spectrum in the relevant range. This is the main idea
behind pseudopotential methods. Specifically, pseudopotentials are designed to match the eigenvalues of the
original atomic model in a fixed energy range. So when used in molecular or solid-state simulations, it seems
reasonable to hope that they will accurately approximate the sought eigenvalues. Pseudopotentials can also be
introduced to take into account relativistic effects at lower cost \cite{dolg2011relativistic}. Thus a wide range
of pseudopotentials with different properties have been developped, among which, Troullier-Martins
\cite{troullier1991efficient} and Kleinman-Bylander \cite{kleinman1982efficacious} norm conserving
pseudopotentials, Vanderbilt \cite{vanderbilt90} ultrasoft pseudopotentials and Goedecker
\cite{goedecker1996separable} pseudopotentials. A mathematical study of the generation of optimal norm-conserving
pseudopotentials for the reduced Hartree-Fock theory and the local density approximation to the density-functional theory
has already been achieved \cite{cances2016existence}. So far as we know, it is the first mathematical result on pseudopotentials. 

Another strategy is to use a better suited basis set. This is the spirit of the augmented plane-wave (APW) method \cite{singh2006planewaves, koelling1975use}. The APW basis functions are discontinuous: inside balls centered at each nucleus, a basis function reproduces the cusp of atomic electronic wave functions and outside these balls, it is a plane-wave. One thus tries to get the best of both worlds: having the singularity behavior encoded in the basis functions and at the same time, having the plane-wave convergence property outside the singularity region. This method can be viewed as a discontinuous Galerkin method and its mathematical analysis has been carried out in \cite{chen2015numerical}.

The projector augmented-wave (PAW) \cite{blochl94} method relies on the same idea as the APW method but here,
instead of using another set of basis functions, the eigenvalue problem is modified by an invertible transformation
which carries the cusp behavior and/or fast oscillations in the vicinities of the nuclei. Specifically, the operator acts locally in a
ball around each nucleus and maps atomic wave functions to smooth functions called pseudo wave functions. The form
of the transformation is compatible with slowly varying pseudopotentials in the Hamiltonian. Doing so, one replaces the Coulomb
potential by a smooth potential without changing the spectrum of the original operator. The solution of the
corresponding \emph{generalized} eigenvalue problem can then be advantageously expanded in plane-waves. Because of
the one-to-one correspondance between the pseudo and the actual wave functions and its efficiency to produce
accurate results, the PAW method has become a very popular tool and has been implemented in several popular electronic structure
simulation codes (AbInit \cite{torrent2008337}, VASP \cite{kresse1996efficient,kresse99}). 
%The pseudo wave function satisfies a \emph{generalized} eigenvalue problem and is expanded in plane-waves. Because of this one-to-one correspondance between the pseudo and the actual wave functions and its efficiency to produce accurate results, the PAW method has become a very popular tool and has been implemented in different molecular codes (AbInit \cite{torrent2008337}, VASP \cite{kresse1996efficient,kresse99}).

A crucial assumption in the PAW method is the completeness of the basis of atomic wave functions used to build the
PAW transformation. In practice, infinite expansions appearing in the generalized eigenvalue problem are truncated,
introducing an error which is rarely quantified. In the variational PAW method (which will be referred to as VPAW
in the following), a finite number of functions is used right from the beginning, avoiding this truncation
error. Although pseudopotentials can no longer be incorporated, an acceleration of convergence is obtained. This acceleration can be precisely characterized in the case of the double Dirac potential in a one-dimensional periodic setting.
 
%A close inspection of the PAW method reveals that the invertible operator is constructed using an infinity of functions which gives operators with infinte sums that have to be truncated in practice. Doing so, the PAW method introduces an error which is rarely quantified. In the variational PAW method (which will be referred to as VPAW in the following), a finite number of functions is used right at the beginning avoiding this truncation error. Although pseudopotentials can no longer be incorporated, an acceleration of convergence is gained which can be precisely caracterized in the case of the double Dirac potential in a one-dimensional periodic setting.

Before moving to the 1D model we analyzed, we will briefly introduce both PAW and VPAW methods applied to a 3D Hamiltonian and explain the major differences between both approaches. More detailed expositions of the PAW method in various settings can be found in \cite{audouze2006projector, jollet2014generation, rostgaard2009projector}. 
\\
\section{PAW vs VPAW methods}
\label{sec:vari-paw-meth}

\subsection{General setting}

The general setting of the VPAW method for finite molecular systems has been presented in \cite{blanc2017}. In this paper, we will focus on the derivation of the VPAW equations in the periodic setting. For simplicity, we restrict ourselves to a linear model. 
%A quick overview of the spectral theory of periodic hamiltonians can be found in \cite{gontier2015mathematical}. More thoroughful expositions of this theory are in \cite{eastham1973spectral,kuchment2012floquet}. 
%For extensions to nonlinear equations, the interested reader is
%referred to \cite{catto2002some} for the Hartree model and to \cite{catto2001thermodynamic} for the Hartree-Fock model. 

The crystal is modelled as an infinite periodic motif of $N_{at}$ point charges at positions $\mathbf{R}_I$ in the unit cell 
$$
\Gamma = \left\{\alpha_1 \mathbf{a}_1 + \alpha_2 \mathbf{a}_2 + \alpha_3 \mathbf{a}_3, (\alpha_1, \alpha_2, \alpha_3) \in [-1/2, 1/2)^3 \right\},
$$
 and repeated over the periodic lattice 
$$
\mathcal{R} = \Z \mathbf{a}_1 + \Z  \mathbf{a}_2 + \Z  \mathbf{a}_3,
$$
where $\mathbf{a}_1,\mathbf{a}_2,\mathbf{a}_3$ are linearly independent vectors of $\R^3$.  

In the linear model under consideration, the electronic properties of the crystal are encoded in the spectral properties of the periodic Hamiltonian $H_\mathrm{per}$ acting on $L^2(\R^3)$:
$$
H_\mathrm{per} = -\frac{1}{2} \Delta + V_\mathrm{per} + W_\mathrm{per} ,
$$
where $V_\mathrm{per}$ is the $\mathcal{R}$-periodic potential defined (up to an irrelevant addition constant) by 
\begin{equation}
\label{eq:Vper}
\begin{cases}
-\Delta V_\mathrm{per} = 4\pi \left( \displaystyle{\sum_{\mathbf{T} \in \mathcal{R}}} \displaystyle{\sum_{I=1}^{N_{at}}} Z_I \left(\delta_{\mathbf{R}_I} (\cdot + \mathbf{T}) - \frac{1}{|\Gamma|}\right) \right) \\
V_\mathrm{per} \text{ is } \mathcal{R}\text{-periodic}.
\end{cases}
\end{equation}
For simplicity, $W_\mathrm{per}$ is a regular enough $\mathcal{R}\text{-periodic}$ potential. In practice, $W_\mathrm{per}$ is a nonlinear potential depending on the model chosen to describe the Hartree and exchange-correlation terms (typically a Kohn-Sham LDA potential). 
\newline

The standard way to study the spectral properties of $H_\mathrm{per}$ is through Bloch theory which will be outlined in the next few lines. Let $\mathcal{R}^*$ be the dual lattice 
$$
\mathcal{R}^* = \Z \mathbf{a}_1^* + \Z \mathbf{a}_2^* + \Z \mathbf{a}_3^*, 
$$
where  $(\mathbf{a}_1^*, \mathbf{a}_2^*,\mathbf{a}_3^*)$ satisfies $\mathbf{a}_i \cdot \mathbf{a}_j^* = 2\pi \delta_{ij}$. The reciprocal unit cell is defined by 
$$
\Gamma^* = \left\{\alpha_1 \mathbf{a}_1^* + \alpha_2 \mathbf{a}_2^* + \alpha_3 \mathbf{a}_3^*, (\alpha_1, \alpha_2, \alpha_3) \in [-1/2, 1/2)^3 \right\}.
$$
As $H_\mathrm{per}$ commutes with $\mathcal{R}$-translations, $H_\mathrm{per}$ admits a Bloch decomposition in operators $H_\mathbf{q}$ acting on 
$$
L^2_\mathrm{per}(\Gamma) = \{ f \in L^2_\mathrm{loc}(\R^3) \ | \ f \ \text{is } \mathcal{R}\text{-periodic}\},
$$ 
with domain 
$$
H^2_\mathrm{per}(\Gamma) = \{ f \in H^2_\mathrm{loc}(\R^3) \ | \ f \ \text{is } \mathcal{R}\text{-periodic}\}.
$$
The operator $H_\mathbf{q}$ is given by:
$$
H_\mathbf{q} = \frac{1}{2} |-i\nabla + \mathbf{q}|^2 + V_\mathrm{per} + W_\mathrm{per}, \quad \mathbf{q} \in \Gamma^*.
$$
For each $\mathbf{q} \in \Gamma^*$, the operator $H_\mathbf{q}$ is self-adjoint, bounded below and with compact resolvent. It thus has a discrete spectrum. Denoting by $E_{1,\mathbf{q}} \leq E_{2,\mathbf{q}} \leq \dots,$ with $E_{n, \mathbf{q}} \underset{n\to+\infty}{\longrightarrow} +\infty$, its eigenvalues counted with multiplicities, there exists an orthonormal basis of $L^2_\mathrm{per}(\Gamma)$ consisting of eigenfunctions $(\psi_{n, \mathbf{q}})_{n \in \N^*}$ 
\begin{equation}
\label{eq:H_per}
H_\mathbf{q} \psi_{n,\mathbf{q}} = E_{n,\mathbf{q}} \psi_{n,\mathbf{q}}.
\end{equation}
The spectrum of $H_\mathrm{per}$ is purely continuous and can be recovered from the discrete spectra of all the operators $H_\mathbf{q}$, $\mathbf{q} \in \Gamma^*$
$$
\sigma (H_\mathrm{per}) = \bigcup_{\mathbf{q} \in \Gamma^*} \sigma ( H_\mathbf{q}). 
$$
%By Bloch theorem, the spectrum of $H$ is given by the union of the discrete spectra of an infinite number of eigenvalue problems parameterized by a vector $\mathbf{q}$ belonging to the reciprocal unit cell $\Gamma^*$ :
%$$
%\Gamma^* = \left\{\alpha_1 \mathbf{a}_1^* + \alpha_2 \mathbf{a}_2^* + \alpha_3 \mathbf{a}_3^*, (\alpha_1, \alpha_2, \alpha_3) \in \left[-\frac{1}{2}, \frac{1}{2}\right)^3 \right\}
%$$ 
%where $(\mathbf{a}_1^*, \mathbf{a}_2^*,\mathbf{a}_3^*)$ is defined by $\mathbf{a}_i \cdot \mathbf{a}_j^* = 2\pi \delta_{ij}$. 
%The eigenvalue problems to solve to reconstruct the original spectrum are
%\begin{equation}
%\label{eq:H_per}
%H_\mathbf{q} \psi_{n,\mathbf{q}} = E_{n,\mathbf{q}} \psi_{n,\mathbf{q}}, 
%\end{equation}
%with $H_\mathbf{q}$ a self-adjoint, bounded below operator acting on $L^2_\mathrm{per}(\Gamma)$ with domain $H^2_\mathrm{per}(\Gamma)$ defined by:
%$$
%H_\mathbf{q} = \frac{1}{2} |-i\nabla + \mathbf{q}|^2 + V_\mathrm{per}
%$$
%and where $(E_{n,\mathbf{q}})_{n \in \N^*}$ are the eigenvalues of $H_\mathbf{q}$, ranked in increasing order and $(\psi_{n,\mathbf{q}})_{n \in \N^*} \in L^2_{\mathrm{per}}(\Gamma)$ the corresponding orthonormal eigenvectors. 

The PAW and VPAW methods aim to ease the numerical approximation of the eigenvalue problem \eqref{eq:H_per}. For clarity, we will only present the case $\mathbf{q}=0$ and denote $H_0$ by $H$ as this special case encloses all the main difficulties encountered in numerically solving \eqref{eq:H_per}. Transposition to $\mathbf{q} \not= 0$ can be done without problem.

\subsection{The VPAW method for solids}

Following the idea of the PAW method, an invertible transformation $(\mathrm{Id}+T)$ is applied to the eigenvalue problem \eqref{eq:H_per}, where $T$ is the sum of operators $T_I$, each $T_I$ acting locally around nucleus $I$. For each operator $T_I$, two parameters $N_I$ and $r_c$ need to be fixed:
\begin{enumerate}
\item the number $N_I$ of PAW functions used to build $T_I$,
\item a cut-off $r_c$ radius which will set the acting domain of $T_I$, more precisely: 
\begin{itemize}
\item for all $f \in L_\mathrm{per}^2(\Gamma)$, $\mathrm{supp}(T_If) \subset \displaystyle{\bigcup_{\mathbf{T} \in \mathcal{R}}} \overline{B}(\mathbf{R}_I + \mathbf{T},r_c)$, where $\overline{B}(\mathbf{R},r)$ is the closed ball of $\R^3$ with center $\mathbf{R}$ and radius $r$,
\item if $\mathrm{supp}(f) \bigcap \bigcup\limits_{\mathbf{T} \in \mathcal{R}} \overline{B}(\mathbf{R}_I + \mathbf{T},r_c) = \emptyset$, then $T_If = 0$.
\end{itemize} 
\end{enumerate}
The cut-off radius $r_c$ must be chosen small enough to avoid pairwise overlaps of the balls $(\overline{B}(\mathbf{R}_I + \mathbf{T}, r_c))_{1 \leq I \leq N, \mathbf{T} \in \mathcal{R}}$. 

The operator $T_I$ is given by:
\begin{equation}
\label{eq:T_0}
T_I = \sum_{k=1}^{N_I} (\phi^I_{k}(\mathbf{r} - \mathbf{R}_I) - \widetilde{\phi}^I_{k}(\mathbf{r} - \mathbf{R}_I)) \psh{\tilde{p}^I_{k}(\cdot - \mathbf{R}_I)}{\bullet},
\end{equation}
where $\psh{\bullet}{\bullet}$ is the usual $L^2_\mathrm{per}$-scalar product and the functions $\phi^I_{k}$, $\widetilde{\phi}^I_k$ and $\tilde{p}^I_k$ are functions in $L^2_\mathrm{per}(\Gamma)$. 

These functions, which will be referred to as the PAW functions in the sequel, are chosen as follows:
\begin{enumerate}
\item first, let $(\varphi^I_{k})_{1 \leq k \leq N_I} \in (L^2(\R^3))^{N_I}$ be eigenfunctions of an atomic \emph{non-periodic} Hamiltonian 
$$
H_I \varphi^I_k = \epsilon_k \varphi^I_k , \quad \epsilon^I_1 \leq \epsilon^I_2 \leq \epsilon^I_3 \leq \dots, \quad \int_{\mathbb{R}^3} \varphi^I_k \varphi^I_{k'} = \delta_{kk'},
$$
with $H_I$ defined by
$$
H_I = -\frac{1}{2} \Delta - \frac{Z_I}{|\mathbf{r}|} + W(|\mathbf{r}|),
$$
where $W$ is a regular enough bounded potential. The operator $H_I$ is
self-adjoint on $L^2(\R^3)$ with domain $H^2(\R^3)$. Again, in practice, $W$ is a radial nonlinear potential belonging to the same family of models as $W_\mathrm{per}$ in Equation \eqref{eq:Vper}. 
The PAW atomic wave functions $(\phi^I_{k})_{1 \leq k \leq N_I} \in (L^2_\mathrm{per}(\Gamma))^{N_I}$ satisfy:
\begin{itemize}
\item for $1 \leq k \leq N_I$ and $\mathbf{r} \in \Gamma$, $\phi^I_k(\mathbf{r}) = \varphi^I_k(\mathbf{r})$, 
\item $\phi^I_k$ is $\mathcal{R}$-periodic;
\end{itemize}
%$$
%H_I = -\frac{1}{2} \Delta - \frac{Z_I}{|\mathbf{r}|} + \left( \rho_I \star \frac{1}{|\cdot|} \right) (\mathbf{r}), \quad \rho_I(\mathbf{r}) = \sumlim{k=1}{Z_I} \left| \phi_k^I (\mathbf{r}) \right|^2.
%$$
\item the pseudo wave functions $(\widetilde{\phi}^I_{k})_{1 \leq k \leq N_I}$, with $N_I \leq Z_I$, are determined by the next conditions:
\begin{enumerate}
\item inside the unit cell $\Gamma$, $\widetilde{\phi}^I_{k}$ is smooth and matches $\phi^I_{k}$ and several of its derivatives on the sphere $\{|\mathbf{r}|=r_c\}$,
\item for $\mathbf{r} \in \R^3 \setminus \bigcup\limits_{\mathbf{T} \in \mathcal{R}} \overline{B}(\mathbf{T}, r_c) $, $\widetilde{\phi}^I_{k}(\mathbf{r}) = \phi^I_{k}(\mathbf{r})$;
\end{enumerate}
\item the projector functions $(\tilde{p}^I_{k})_{1 \leq k \leq N_I}$ are defined such that:
\begin{enumerate}
\item each projector function $\tilde{p}^I_{k}$ is supported in $\bigcup\limits_{\mathbf{T} \in \mathcal{R}} \overline{B}(\mathbf{T},r_c)$,
\item they form a dual family to the pseudo wave functions $(\widetilde{\phi}^I_{k})_{1 \leq k \leq N_I}$: $\psh{\tilde{p}^I_{k}}{\widetilde{\phi}^I_{k'}} = \delta_{kk'}$.
\end{enumerate}
\end{enumerate}
By our choice of the pseudo wave functions $\widetilde{\phi}^I_k$ and the projectors $\tilde{p}^I_k$, $T_I$ acts in \linebreak $\bigcup\limits_{\mathbf{T} \in \mathcal{R}} \overline{B}(\mathbf{R}_I + \mathbf{T},r_c)$. 
\newline

The VPAW equations to solve are then:
\begin{equation}
\label{eq:VPAW_eig_pb}
\widetilde{H} \tilde{\psi} = E \widetilde{S} \tilde{\psi},
\end{equation}
where 
\begin{equation}
\label{eq:H_VPAW}
\widetilde{H} = (\mathrm{Id}+T)^* H (\mathrm{Id}+T), \quad \widetilde{S} = (\mathrm{Id}+T)^*(\mathrm{Id}+T),
\end{equation}
and 
$$
T = \sum_{I=1}^{N_{at}} T_I.
$$
Thus if $(\mathrm{Id}+T)$ is invertible, it is easy to recover the eigenfunctions of $H$ by the formula
\begin{equation}
\label{eq:I+Tpsi}
\psi = (\mathrm{Id}+T) \tilde{\psi},
\end{equation}
and the eigenvalues of $H$ coincide with the generalized eigenvalues of \eqref{eq:VPAW_eig_pb}. 

By construction, the operator $(\mathrm{Id}+T_I)$ maps the pseudo wave functions $\tilde{\phi}$ to the atomic eigenfunctions $\phi$:
$$
(\mathrm{Id}+T_I)\tilde{\phi}^I_{k} = \phi^I _{k},
$$
so if locally around each nucleus, the function $\psi$ "behaves" like the atomic wave functions $\phi_{k}$, we can expect that the cusp behavior of $\psi$ is captured by the operator $T$, thus $\tilde{\psi}$ is smoother than $\psi$ and the plane-wave expansion of $\tilde{\psi}$ converges faster than the expansion of $\psi$. \\

\subsection{Differences with the PAW method}
\label{sec:HPAW-pseudo}

The PAW equations solved in materials science simulation packages are different from the VPAW equations \eqref{eq:VPAW_eig_pb}. As in \cite{blochl94}, the construction of $T_I$ involves "complete" infinite sets of functions $\phi^I_k$, $\widetilde{\phi}^I_k$ and $\tilde{p}^I_k$ in the sense that for a function $f \in L^2_\mathrm{per}(\Gamma)$ supported in the balls $\bigcup\limits_{\mathbf{T} \in \mathcal{R}} \overline{B}(\mathbf{R}_I + \mathbf{T},r_c)$, we have :
\begin{equation}
\label{eq:PAW_completeness}
\sumlim{k=1}{\infty} \psh{\tilde{p}^I_k}{f} \widetilde{\phi}^I_k(x) = f(x), \quad \text{for a.a. } x \in \bigcup\limits_{\mathbf{T} \in \mathcal{R}} \overline{B}(\mathbf{R}_I + \mathbf{T},r_c).
\end{equation}

This relation enables one to simplify the expression of $(\mathrm{Id}+T^*)H(\mathrm{Id}+T)$ and \linebreak $(\mathrm{Id}+T^*)(\mathrm{Id}+T)$ to 
\begin{multline}
\label{eq:H_fullPAW}
H^{PAW} = (\mathrm{Id}+T^*) H (\mathrm{Id}+T) = H + \sumlim{I=1}{N_{at}} \sumlim{i,j=1}{\infty} \tilde{p}^I_i(\cdot - \mathbf{R}_I) \\
\left( \psh{\phi^I_i(\cdot - \mathbf{R}_I)}{H \phi^I_j(\cdot - \mathbf{R}_I)} - \psh{\widetilde{\phi}^I_i(\cdot - \mathbf{R}_I)}{H \widetilde{\phi}^I_j(\cdot - \mathbf{R}_I)} \right) \psh{\tilde{p}^I_j(\cdot - \mathbf{R}_I)}{\bullet} ,
\end{multline} 
and
\begin{multline}
\label{eq:S_fullPAW}
S^{PAW} = (\mathrm{Id}+T^*) (\mathrm{Id}+T) = \mathrm{Id} + \sumlim{I=1}{N_{at}} \sumlim{i,j=1}{\infty} \tilde{p}^I_i(\cdot - \mathbf{R}_I) \\
\left( \psh{\phi^I_i(\cdot - \mathbf{R}_I)}{ \phi^I_j(\cdot - \mathbf{R}_I)} - \psh{\widetilde{\phi}^I_i(\cdot - \mathbf{R}_I)}{ \widetilde{\phi}^I_j(\cdot - \mathbf{R}_I)} \right) \psh{\tilde{p}^I_j(\cdot - \mathbf{R}_I)}{\bullet} .
\end{multline} 

In practice, the double sums on $i,j$ appearing in the operators $H^{PAW}$ and $S^{PAW}$ are then truncated and the so-obtained
generalized eigenvalue problem is solved. Thus the identity
$\psi=(\mathrm{Id}+T)\tilde{\psi}$ does not hold anymore and the eigenvalues of the truncated problem are different from the
exact ones. An analysis of the eigenvalue problem $H^{PAW} f = E S^{PAW}f$ in our 1D-toy model will be provided in another paper \cite{blanc2017paw}.

In contrast, the VPAW approach makes use of a finite number of wave functions
$\phi^I_i$ right from the beginning, avoiding truncation approximations.
\newline

A further modification is used in practice. As the pseudo wave functions $\widetilde{\phi}_k^I$ are equal to $\phi_k^I$ outside the balls $\bigcup\limits_{\mathbf{T} \in \mathcal{R}} \overline{B}(\mathbf{R}_I + \mathbf{T},r_c)$, the integrals appearing in \eqref{eq:H_fullPAW} can be truncated to the ball $\overline{B}(\mathbf{R}_I,r_c)$. Doing so, another expression of $H^{PAW}$ can be obtained :
\begin{multline*}
H^{PAW} = (\mathrm{Id}+T^*) H (\mathrm{Id}+T) = H + \sumlim{I=1}{N_{at}} \sumlim{i,j=1}{\infty} \tilde{p}^I_i(\cdot - \mathbf{R}_I) \\
\left( \psh{\phi^I_i(\cdot - \mathbf{R}_I)}{H \phi^I_j(\cdot - \mathbf{R}_I)}_{\mathbf{R}_I} - \psh{\widetilde{\phi}^I_i(\cdot - \mathbf{R}_I)}{H \widetilde{\phi}^I_j(\cdot - \mathbf{R}_I)}_{\mathbf{R}_I} \right) \psh{\tilde{p}^I_j(\cdot - \mathbf{R}_I)}{\bullet} ,
\end{multline*}
where 
\begin{equation*}
\psh{f}{g}_{\mathbf{R}_I} = \itg{B(\mathbf{R}_I,r_c)}{}{f(x)g(x)}{x}.
\end{equation*}
Using this expression of the operator $H^{PAW}$, it is possible to introduce an $\mathcal{R}$-periodic potential $V^\mathrm{PP}$ such that :
\begin{enumerate}
\item $V^\mathrm{PP} = V_\mathrm{per}$ outside $\bigcup\limits_{\mathbf{T} \in \mathcal{R}} \bigcup\limits_{I=1}^{N_{at}} \overline{B}(\mathbf{R}_I + \mathbf{T},r_c)$,
\item $V^\mathrm{PP}$ is smooth inside $\bigcup\limits_{\mathbf{T} \in \mathcal{R}} \bigcup\limits_{I=1}^{N_{at}} \overline{B}(\mathbf{R}_I + \mathbf{T},r_c)$.
\end{enumerate}
%Usually the pseudopotential $V^\mathrm{PP}$ is expressed as a sum of atomic pseudopotential $V^\mathrm{PP}$ smoothed around each nucleus $I$:
%\begin{equation}
%V^\mathrm{PP} = \sumlim{I=1}{N_{at}} V^\mathrm{PP}_I.
%\end{equation}

The expression of $H^{PAW}$ is
\begin{multline}\label{eq:1}
H^{PAW} = H_\mathrm{ps} + \sumlim{I=1}{N_{at}} \sumlim{i,j=1}{\infty} \tilde{p}^I_i(\cdot - \mathbf{R}_I) \\
\left( \psh{\phi^I_i(\cdot - \mathbf{R}_I)}{H \phi^I_j(\cdot - \mathbf{R}_I)}_{\mathbf{R}_I} - \psh{\widetilde{\phi}^I_i(\cdot - \mathbf{R}_I)}{H_\mathrm{ps} \widetilde{\phi}^I_j(\cdot - \mathbf{R}_I)}_{\mathbf{R}_I} \right) \psh{\tilde{p}^I_j(\cdot - \mathbf{R}_I)}{\bullet} ,
\end{multline}
with $H_{ps}$:
\begin{displaymath}
H_\mathrm{ps} = - \frac{1}{2} \Delta + V^\mathrm{PP},
\end{displaymath}
%\begin{equation}
%H_\mathrm{ps} = -\frac{1}{2} \Delta + \sumlim{I=1}{N_{at}} V_{I}^\mathrm{PP},
%\end{equation}
where $V^\mathrm{PP}$ is a smooth pseudopotential. Note that, in practice, the sum of $i,j$ in \eqref{eq:1} is
truncated to some level $N_I$.

A study of the associated eigenvalue problem will also be provided in the aforementioned paper \cite{blanc2017paw}. Some numerical tests comparing the VPAW and PAW methods in our one-dimensional toy model can be found in Section \ref{sec:numerics}.

\subsection{Computational complexity}

A detailed analysis of the computational cost of the PAW method can be found in \cite{levitt2015parallel}: the cost scales like
$\cO(N M + M \log M)$ where $N = \sum_I N_I$ is the total number of projectors and $M$ the number of plane-waves. Usually, $N$ is
chosen relatively small, but $M$ may be large, so it is important to avoid a computational cost of order $M^2$. 

%We now give the same analysis for VPAW, with a similar result.
In practice, we are interested in the cost of the computation of $\widetilde{H} \tilde{\psi}$ and $\widetilde{S} \tilde{\psi}$ where $\tilde{\psi}$ is expanded in $M$ plane-waves as the generalized eigenvalue problem is solved by a conjugate gradient algorithm. We will only focus on $\widetilde{H} \tilde{\psi}$ since the analysis $\widetilde{S} \tilde{\psi}$ is similar. Let us split $\widetilde{H}$ into four terms:
$$
\widetilde{H} \tilde{\psi} = H \tilde{\psi} + P D_H P^T \tilde{\psi} + H \left(\Phi - \widetilde{\Phi}\right) P^T \tilde{\psi} + P H \left(\Phi - \widetilde{\Phi}\right)^T \tilde{\psi} ,
$$
where $P$ is the $M \times N$ matrix of the projector functions, $H (\Phi - \widetilde{\Phi})$ the $M \times N$ matrix of the Fourier representation of the $N$ functions $H(\phi_i - \widetilde{\phi}_i)$, and $D_H$ is the $N \times N$ matrix $\psh{\phi_i - \widetilde{\phi}_i}{H (\phi_j - \widetilde{\phi}_j)}$. 

The computational cost can be estimated as follows (the cost at each step is given in brackets):
\begin{enumerate}
\item $H \tilde{\psi}$ is assembled in two steps. First, $-\frac{1}{2} \Delta \tilde{\psi}$ is computed in $\cO (M)$ since the operator $\frac{1}{2} \Delta$ is diagonal in Fourier representation. For the potential $V$, apply an inverse FFT to $\tilde{\psi}$ to have the real space representation of $\tilde{\psi}$, multiply pointwise by $V$ and apply a FFT to the whole result ($\cO(M \log M)$);
\item for $P D_H P^T \tilde{\psi}$, compute the $N$ projections $P^T \tilde{\psi}$ ($\cO(MN)$), then successively apply the matrices $D_H$ ($\cO(N^2)$) and $P$ ($\cO(MN)$);
\item for $P H (\Phi - \widetilde{\Phi})^T \tilde{\psi}$, similarly apply successively $H(\Phi - \widetilde{\Phi})^T$ to $\tilde{\psi}$ ($\cO(MN)$) and $P$ to $H (\Phi - \widetilde{\Phi})^T \tilde{\psi}$ ($\cO(MN)$);
\item for $H (\Phi - \widetilde{\Phi}) P^T \tilde{\psi}$, we proceed as in step 3.
\end{enumerate}
Thus, the total numerical cost is of order $\cO(MN + M \log M)$ which is the same as for the PAW method. 

The matrix $H
(\Phi - \widetilde{\Phi})$ is approximated by a plane-wave expansion, which may be a poor
approximation because of the singularities of $\Phi$. However, it should be noticed that this is only an
intermediary in the computation of $\tilde\psi$, which is well approximated by plane-waves. Hence it is not
clear that a poor approximation of  $H(\Phi - \widetilde{\Phi})$ should imply a poor approximation of $\tilde \psi.$

% The matrix $H(\Phi - \widetilde{\Phi})$ is a poor representation of the functions $H(\phi_i - \widetilde{\phi}_i)$ thus a bad discretisation of the operator $\widetilde{H}$, however, $\widetilde{H}$ is discretised because the function $\tilde{\psi}$ is expanded in a plane-wave basis. The function $\tilde{\psi}$ can be accurately approximated with a few plane-waves, so this poor representation should not be a major concern. 

%The computational cost can be estimated as follows (the cost at each step is given between parentheses):
%\begin{enumerate}
%\item $H \tilde{\psi}$ is assembled in two steps. Decomposing $H$ into $K$, the kinetic part and $V$ a local potential, $K \tilde{\psi}$ is obtained by a simple scaling then apply an inverse FFT to $\tilde{\psi}$ to have the real space representation of $\tilde{\psi}$, multiply by $V$ and apply a FFT on the whole ($O(M \log M)$).
%\item for $P D_H P^T \tilde{\psi}$, compute the $N$ projections $P^T \tilde{\psi}$ ($O(MN)$), then successively apply the matrix $D_H$ ($O(N^2)$) and $P$ ($O(MN)$).
%\item for $P H (\Phi - \widetilde{\Phi})^T \tilde{\psi}$, similarly apply successively $H(\Phi - \widetilde{\Phi})^T$ to $\tilde{\psi}$ ($O(MN)$) and $P$ to $H (\Phi - \widetilde{\Phi})^T \tilde{\psi}$ ($O(MN)$).
%\item for $H (\Phi - \widetilde{\Phi}) P^T \tilde{\psi}$, we proceed as in (iii).
%\end{enumerate}

\subsection{Generation of the pseudo wave functions}

In practice, there are two main ways to generate the pseudo wave functions $\tilde{\phi}_{k}$ and the projectors $\tilde{p}_{k}$ introduced by Blöchl \cite{blochl94} and Vanderbilt \cite{laasonen93}. For both schemes, the generation of the PAW functions has to be done for \emph{each} angular momentum $\ell$.

\subsubsection{Vanderbilt scheme}

\paragraph{Atomic wave function}

The functions $\varphi_k$ are simply the atomic wavefunctions defined earlier \emph{i.e.} solutions to the atomic eigenvalue problem 
$$
H_I \varphi^I_k = \epsilon_k \varphi^I_k , \quad \epsilon^I_1 \leq \epsilon^I_2 \leq \epsilon^I_3 \leq \dots, \quad \int_{\mathbb{R}^3} \varphi^I_j \varphi^I_k = \delta_{j k},
$$
with
$$
H_I = -\frac{1}{2} \Delta - \frac{Z_I}{|\mathbf{r}|} + W(|\mathbf{r}|).
$$
%$$
%H_I = -\frac{1}{2} \Delta - \frac{Z_I}{|\mathbf{r}|} + \left( \rho_I \star \frac{1}{|\cdot|} \right) (\mathbf{r}), \quad \rho_I(\mathbf{r}) = \sumlim{k=1}{Z_I} \left| \phi_k^I (\mathbf{r}) \right|^2.
%$$
%It has been proved for reduced Hartree-Fock models \cite{Solovej1991} that $\phi_k$ can be be decomposed into a spherical part $Y_{\ell m}$ -the Laplace spherical harmonics- and a radial part $R_{n \ell}$:
The eigenfunctions $\varphi_k$ can be decomposed into a spherical part $Y_{\ell m}$ -the real Laplace spherical harmonics- and a radial part $R_{n \ell}$
$$
\varphi_{k}(\br) = \frac{R_{n \ell}(r)}{r} Y_{\ell m} (\omega),
$$
where $k$ stands for the multiple indices $(n,\ell,m)$ and $\br = (r,\omega)$ is written in polar coordinates. 

\paragraph{Pseudo wave function}

The pseudo wave functions $\widetilde{\phi}_k$ are given by :
$$
\forall \br \in \Gamma, \ \widetilde{\phi}_k(\br) = \frac{\widetilde{R}_{n \ell}(r)}{r}\, Y_{\ell m}(\omega).
$$
Various choices of $\widetilde{R}_{n \ell}$ are possible, for example,  in \cite{laasonen93}, $\widetilde{R}_{n \ell}$ is a polynomial inside the augmentation region $\overline{B}(0,r_c)$ :
$$
\widetilde{R}_{n \ell}(r) = 
\begin{cases}
r^{\ell+1} \sumlim{k=0}{p} c_{2k} r^{2k}  & \text{for } 0 \leq r \leq r_c,\\
R_{n \ell}(r) & \text{for } r > r_c,
\end{cases}
$$
or in \cite{kresse99}, a sum of spherical Bessel functions $j_\ell$ :
$$
\widetilde{R}_{n \ell}(r) = 
\begin{cases}
r \sumlim{k=1}{p} \alpha_k j_\ell(q_k r) \ ,  & \text{for } 0 \leq r \leq r_c,\\
R_{n \ell}(r) & \text{for } r > r_c,
\end{cases}
$$
and the coefficients are chosen to match as many derivatives of $R_{n \ell}$ as possible at $r_c$.

\paragraph{Projector function}

First, define :
$$
\chi_{n \ell}(r) = \frac{1}{2} \widetilde{R}_{n \ell}''(r) + \left( E_n - \frac{\ell (\ell +1)}{2r^2}  - V^\mathrm{PP}_\ell(r) \right) \widetilde{R}_{n \ell}(r),
$$
where $V^\mathrm{PP}_\ell(r)$ is usually the Troullier-Martins pseudopotential \cite{troullier1991efficient} although other choices are possible. 

By construction, $\mathrm{supp}(\chi_{n \ell}) \subset [0,r_c]$. Let $B$ be the matrix 
$$
B_{n , n'} = \itg{0}{r_c}{\widetilde{R}_{n \ell}(r) \chi_{n' \ell}(r)}{r}.
$$
The radial parts of the projector functions are given by
$$
p_{n \ell}(r) =  \sumlim{n'=1}{N_I} \chi_{n' \ell}(r) \left( B^{-1} \right)_{n' n}.
$$
The projector functions are defined by
$$
\widetilde{p}_{n \ell}(\br) = \frac{p_{n \ell}(r)}{r} Y_{\ell m}(\omega).
$$
This ensures that $\psh{\widetilde{p}_{n \ell}}{\widetilde{\phi}_{n' \ell'}} = \delta_{nn'}\delta_{\ell \ell'}$. 

\subsubsection{Blöchl scheme}

The PAW functions are generated in two steps. For each angular momentum $\ell$, we define auxiliary functions $\widetilde{R}^0_{n \ell}$ and $p^0_{n \ell}$:

\paragraph{Auxiliary functions}

Let $\chi(r)$ be the cut-off function 
$$
\chi(r) = 
\begin{cases}
 \frac{\sin (\pi r /r_c)}{(\pi r /r_c)} & \text{for } r \leq r_c, \\
0 & \text{for } r > r_c,
\end{cases}
$$
and let $(C_{n \ell}, \widetilde{R}^0_{n \ell})_{n \in \mathbb{N}^*}$ be the unique solution to:
\begin{equation}
\label{eq:eig_blochl}
\begin{cases}
- \frac{1}{2} (\widetilde{R}^0_{n \ell}) ^{\prime \prime}(r) + \frac{\ell (\ell +1)}{2r^2} \widetilde{R}^0_{n \ell} + (V^\mathrm{PP}_\ell - E_n) \widetilde{R}^0_{n \ell} = C_{n \ell} \chi(r) \widetilde{R}^0_{n \ell} , \quad 0 \leq r \leq r_c \\
\widetilde{R}^0_{n \ell}(0) = 0 \\
\widetilde{R}^0_{n \ell}(r_c) = R^0_{n \ell}(r_c), \quad (\widetilde{R}^0_{n \ell} )^\prime (r_c) = (R^0_{n \ell})^\prime(r_c).
\end{cases}
\end{equation}
Let $p^0_{n \ell}$ be the auxiliary functions:
$$
p^0_{n \ell}(r) = \frac{\chi(r) \widetilde{R}^0_{n \ell}(r)}{\left(\widetilde{R}^0_{n \ell}(r) | \chi(r) \widetilde{R}^0_{n \ell}(r) \right)},
$$
where 
$$
\left( f | g \right) = \itg{0}{r_c}{f(r)g(r)}{r}.
$$

\paragraph{PAW functions}

Finally the radial part of all the PAW functions are constructed with a Gram-Schmidt process. We describe it here
assuming that only two quantum numbers $n_1<n_2$ are needed for the computation. However, one should bear in mind
that, on the one hand, usually, $n_2 = n_1+1$, and on the other hand, although limiting the procedure to two quantum numbers is in general sufficient for practical purposes, it is straightforward to generalize the following orthogonalization procedure to an arbitrary
number of quantum numbers. 
\begin{enumerate}
\item \emph{Basis:} the first set of functions $\widetilde{R}_{n_1 \ell}$, $p_{n_1 \ell}$ and $R_{n_1 \ell}$, corresponding to the lowest principal quantum number $n$ used, are defined by
$$
\widetilde{R}_{n_1 \ell} = \widetilde{R}^0_{n_1 \ell}, \quad p_{n_1 \ell} = p^0_{n_1 \ell}, \quad R_{n_1 \ell} = R^0_{n_1 \ell}.
$$
\item \emph{Inductive step:}  if there is a second radial basis function for $n_2 > n_1$, 
	\begin{itemize}
		\item first, the function $\widetilde{R}^0_{n_2 \ell}$ is orthogonalized against $p_{n_1 \ell}$:
		\begin{equation}
		\label{eq:blochl_tdR}
		\widetilde{R}_{n_2 \ell}(r) =\mathcal{F}_{n_2 \ell} \left( \widetilde{R}_{n_2 \ell}^0(r) - \widetilde{R}_{n_1 \ell}(r) \left( p_{n_1 \ell} | \widetilde{R}_{n_2 \ell}^0 \right) \right),
		\end{equation}
		where the factor 
		\begin{equation*}
	\mathcal{F}_{n_2 \ell} = \frac{1}{\left( 1 - \left( \widetilde{R}_{n_2 \ell}^0 | p_{n_1 \ell} \right) \left( \widetilde{R}_{n_1 \ell}^0 | p_{n_2 \ell} \right)  \right)^{1/2}}
		\end{equation*}
		is a normalization constant;
		\item similarly, the function $p^0_{n_2 \ell}$ is orthogonalized against $\widetilde{R}_{n_1 \ell}$ by noticing that \linebreak $\left( \widetilde{R}_{n_2 \ell}^0 | p_{n_1 \ell} \right) = \left( \widetilde{R}_{n_1 \ell}^0 | p^0_{n_2 \ell} \right)$:
		\begin{equation*}
%		\label{eq:blochl_p}
		p_{n_2 \ell}(r) = \mathcal{F}_{n_2 \ell} \left( p^0_{n_2 \ell} - p_{n_1 \ell} \left( \widetilde{R}_{n_2 \ell}^0 | p_{n_1 \ell} \right) \right);
		\end{equation*}
		\item finally, to ensure the continuity between the radial functions ${R}_{n_2 \ell}$ and $\widetilde{R}_{n_2 \ell}$, we apply to ${R}_{n_2 \ell}^0$ the same linear combination in Equation \eqref{eq:blochl_tdR} 
		\begin{equation*}
%		\label{eq:blochl_R} 
		{R}_{n_2 \ell}(r) =\mathcal{F}_{n_2 \ell} \left( {R}_{n_2 \ell}^0(r) - {R}_{n_1 \ell}(r) \left( p_{n_1 \ell} | \widetilde{R}_{n_2 \ell}^0 \right) \right).
		\end{equation*}
	\end{itemize}
%the second set of PAW functions are given by
%\begin{align}
%\widetilde{R}_{n_2 \ell}(r) & =\mathcal{F}_{n_2 \ell} \left( \widetilde{R}_{n_2 \ell}^0(r) - \widetilde{R}_{n_1 \ell}(r) \left( p_{n_1 \ell} | \widetilde{R}_{n_2 \ell}^0 \right) \right),   \\
%p_{n_2 \ell}(r) & = \mathcal{F}_{n_2 \ell} \left( p^0_{n_2 \ell} - p_{n_1 \ell} \left( \widetilde{R}_{n_2 \ell}^0 | p_{n_1 \ell} \right) \right),   \\
%{R}_{n_2 \ell}(r) & =\mathcal{F}_{n_2 \ell} \left( {R}_{n_2 \ell}^0(r) - {R}_{n_1 \ell}(r) \left( p_{n_1 \ell} | \widetilde{R}_{n_2 \ell}^0 \right) \right), 
%\end{align}
%and 
%\begin{equation*}
%\mathcal{F}_{n_2 \ell} = \frac{1}{\left( 1 - \left( \widetilde{R}_{n_2 \ell}^0 | p_{n_1 \ell} \right) \left( \widetilde{R}_{n_1 \ell}^0 | p_{n_2 \ell} \right)  \right)^{1/2}}.
%\end{equation*}
\end{enumerate}
% Usually, in the PAW method, there is no need to use more than two radial functions per angular momentum but generalization with more radial functions is straightforward. 
%\newline

The PAW functions are given by 
$$
\phi_{n \ell m}(x) = \frac{R_{n \ell}(r)}{r} \, Y_{\ell m}(\theta, \varphi),
$$
$$
\widetilde{\phi}_{n \ell m}(x) = \frac{\widetilde{R}_{n \ell}(r)}{r} \, Y_{\ell m}(\theta, \varphi),
$$
$$
\tilde{p}_{n \ell m}(x) = \frac{p_{n \ell}(r)}{r} \, Y_{\ell m}(\theta, \varphi).
$$
\newline

\paragraph{Organization of the paper}

In this paper, we apply the VPAW formalism to the double Dirac potential with periodic boundary conditions in one dimension. The eigenfunctions of this model have a derivative jump at the positions of the Dirac potentials which is similar to the electronic wave function cusp. Furthermore, the eigenvalues and eigenfunctions being known analytically, it is possible to confront our theoretical results to very accurate numerical tests. 

In Section \ref{sec:setting}, we carefully present the VPAW method in our framework. In Section \ref{sec:results}, Fourier decay estimates of the pseudo wave functions are given as well as estimates on the computed eigenvalues. Proofs of these results are gathered in Section \ref{sec:proof}. In Section \ref{sec:smooth}, we discuss the effect of the addition of a smooth potential to the double Dirac model. Numerical simulations which confirm the obtained theoretical results are provided in Section \ref{sec:numerics}.

\section*{Notation}

From now on, $\psh{\cdot}{\cdot}$ denotes the usual inner product in $L^2_\mathrm{per}(0,1)$. 

Let $f$ be a piecewise continuous function. We denote by:
$$
[f]_x := f(x_+) - f(x_-),
$$
where $f(x_+)$ and $f(x_-)$ are respectively the right-sided and left-sided limits of $f$ at $x$. 

Let $f$ be a continuous function. We denote by 
$$
\underbrace{\idotsint}_{2j+2} f (x) = \int\limits_0^x \int\limits_0^{t_1} \dots \int\limits_0^{t_{2j+1}} f(t_{2j+2}) \, \mathrm{d}t_{2j+2} \dots \mathrm{d}t_1
$$
the $(2j+2)$-primitive function of $f$ vanishing at $0$ as well as its first $(2j+1)$-st derivatives.

For $a$ and $b$ in $\mathbb{R}^N$, $a \cdot b$ is the Euclidean inner product. $e_k$ is the $k$-th canonical vector of $\mathbb{R}^d$ or $\mathbb{R}^N$. $I_N$ is the identity matrix of size $N$.

\section{The VPAW method for a one-dimensional model}
\label{sec:setting}

\subsection{The double Dirac potential}

We are interested in the lowest eigenvalues of the 1-D periodic Schr\"odinger operator $H$  on $L^2_{\mathrm{per}}(0,1) := \{ \phi \in L^2_{\mathrm{loc}} (\mathbb{R}) \ | \ \phi \text{ 1-periodic} \}$ with form domain $H^1_{\mathrm{per}}(0,1) := \{ \phi \in L^2_{\mathrm{loc}} (\mathbb{R}) \ | \ \phi' \in L^2_{\mathrm{loc}} (\mathbb{R})\}$ :
\begin{equation}
H = -\frac{\mathrm{d}^2}{\mathrm{d}x^2} - Z_0 \sumlim{k \in \Z}{} \delta_k - Z_a \sumlim{k \in \Z}{} \delta_{k+a},
\label{eq:H_mol}
\end{equation}
where $0 < a < 1$, $Z_0, Z_a >0$. 

A mathematical analysis of this model has been carried out in \cite{cances2017discretization}. There are two negative eigenvalues $E_0 = -\omega_0^2$ and $E_1 = -\omega_1^2$ which are given by the zeros of the function 
$$
f(\omega) = 2 \omega^2 (1- \cosh (\omega)) + (Z_0 + Z_a) \omega \sinh (\omega) - Z_0 Z_a \sinh(a \omega) \sinh((1-a)\omega).
$$
The corresponding eigenfunctions are 
\begin{equation*}
\psi_k(x) = 
\begin{cases}
A_{1,k} \cosh (\omega_k x) + B_{1,k} \sinh(\omega_k x) \ , \ 0 \leq x \leq a, \\
A_{2,k} \cosh (\omega_k x) + B_{2,k} \sinh(\omega_k x) \ , \ a \leq x \leq 1, \\
\end{cases}
\end{equation*}
where the coefficients $A_{1,k}$, $A_{2,k}$, $B_{1,k}$ and $B_{2,k}$ are determined by the continuity conditions and the derivative jumps at $0$ and $a$.

There is an infinity of positive eigenvalues $E_{k+2} = \omega_{k+2}^2$ which are given by the $k$-th zero of the function :
\begin{equation*}
f(\omega) = 2 \omega^2 (1- \cos (\omega)) + (Z_0 + Z_a) \omega \sin (\omega) + Z_0 Z_a \sin(a \omega) \sin((1-a)\omega),
\end{equation*}
and the corresponding eigenfunctions $H \psi_k = \omega_k^2 \psi_k$ are   
$$
\psi_k(x) = 
\begin{cases}
A_{1,k} \cos ( \omega_k x) + B_{1,k} \sin( \omega_k x) \ , \ 0 \leq x \leq a, \\
A_{2,k} \cos ( \omega_k x) + B_{2,k} \sin( \omega_k x) \ , \ a \leq x \leq 1, \\
\end{cases}
$$
where again the coefficients $A_{1,k}$, $A_{2,k}$, $B_{1,k}$ and $B_{2,k}$ are determined by the continuity conditions and the derivative jumps at $0$ and $a$.

\subsection{The VPAW method}
\label{subsec:VPAW1D}

The principle of the VPAW method consists in replacing the original eigenvalue problem
$$
H\psi = E\psi,
$$
by the generalized eigenvalue problem:
\begin{equation}
(\mathrm{Id}+T)^* H (\mathrm{Id}+T) \tilde{\psi} = E (\mathrm{Id}+T)^* (\mathrm{Id}+T) \tilde{\psi},
\label{eq:H_PAW}
\end{equation}
where $\mathrm{Id}+T$ is an invertible bounded linear operator on $L^2_\mathrm{per}(0,1)$. Thus both problems have the same eigenvalues and it is straightforward to recover the eigenfunctions of the former from the generalized eigenfunctions of the latter:
\begin{equation}
\psi = (\mathrm{Id}+T) \tilde{\psi}.
\label{eq:psi_PAW}
\end{equation}

$T$ is the sum of two operators acting near the atomic sites 
$$
T = T_0 + T_a.
$$
To define $T_0$, we fix an integer $N$ and a radius $0 < \eta < \min(\frac{a}{2}, \frac{1-a}{2})$ so that $T_0$ and $T_a$ act on two disjoint regions $\bigcup\limits_{k \in \Z} [-\eta+k,\eta+k]$ and $\bigcup\limits_{k \in \Z} [a-\eta+k,a+\eta+k]$ respectively. 

\paragraph{Atomic wave function $\phi_k$}
Let $H_0$ be the operator defined by :
$$
H_0 = -\frac{\mathrm{d}^2}{\mathrm{d}x^2} - Z_0 \sumlim{k \in \Z}{} \delta_k .
$$
By parity, the eigenfunctions of this operator are even or odd. The odd eigenfunctions are in fact $x \mapsto \sin (2 \pi k x)$ and the even ones are the $1$-periodic functions such that
$$
\begin{cases}
\phi_0(x) := \cosh(\omega_0 (x - \tfrac{1}{2})), \text{for } x \in [0,1], \\
\phi_k(x) := \cos(\omega_k (x-\tfrac{1}{2}) ), \text{for } x \in [0,1], k \in \N^*.
\end{cases}
$$
To construct $T_0$, we will only select the non-smooth thus even eigenfunctions $(\phi_k)_{1 \leq k \leq N}$ and denote by $(\epsilon_k)_{1 \leq k \leq N}$ the corresponding eigenvalue:
$$
H_0 \phi_k = \epsilon_k \phi_k.
$$

\paragraph{Pseudo wave function $\tilde{\phi}_i$}
The pseudo wave functions $(\tilde{\phi}_i)_{1 \leq i \leq N} \in \left(H^1_\mathrm{per}(0,1)\right)^N$ are defined as follows:
\begin{enumerate}
\item for $|x| \notin \bigcup\limits_{k \in \Z} [-\eta + k, \eta+k]$, $\tilde{\phi}_i(x) = \phi_i(x)$.
\item for $|x| \in \bigcup\limits_{k \in \Z} [-\eta + k, \eta+k]$, $\tilde{\phi}_i$ is an even polynomial of degree at most $2d-2$, $d \geq N$.  
\item $\tilde{\phi}_i$ is $C^{d-1}$ at $\eta$ \emph{i.e.} $\tilde{\phi}_i^{(k)}(\eta) = \phi_i^{(k)}(\eta)$ for $0 \leq k \leq d-1$.
\end{enumerate}

\paragraph{Projector functions $\tilde{p}_i$}
Let $\rho$ be a positive, continuous function with support $[-1,1]$ and $\rho_\eta(t) = \sumlim{k\in \Z}{} \rho(\tfrac{t-k}{\eta})$. The projector functions $(\tilde{p}_i)_{1 \leq i \leq N}$ are obtained by an orthonormalization procedure from the functions $p_i(t) = \rho_\eta (t) \tilde{\phi}_i(t)$ in order to satisfy the duality condition :
$$
\psh{\tilde{p}_i}{\tilde{\phi}_j} = \delta_{ij} .
$$
More precisely, we compute the matrix $B_{ij} := \psh{p_i}{\tilde{\phi}_j}$ and invert it to obtain the projector functions
$$
\tilde{p}_k = \sumlim{j=1}{N} (B^{-1})_{kj} p_j. 
$$
The matrix $B$ is the Gram matrix of the functions $\tilde{\phi}_j$ for the weight $\rho_\eta$. The orthogonalization is possible only if the family $(\tilde{\phi}_i)_{1 \leq i \leq N}$ is independent - thus necessarily $d \geq N$.  
\\

$T_0$ and $T_a$ are given by :
\begin{equation}
\label{eq:T0}
T_0 = \sumlim{i=1}{N} (\phi_{i} - \tilde{\phi}_i) \psh{\tilde{p}_i}{\bullet}, \qquad T_a = \sumlim{i=1}{N} (\phi_{i}^a - \tilde{\phi}_i^a) \psh{\tilde{p}_i^a}{\bullet} \ ,
\end{equation}
where $\phi_{i}^a$ are singular eigenfunctions of the operator $H_a = -\frac{\mathrm{d}^2}{\mathrm{d}x^2} - Z_a \sumlim{k \in \Z}{}\delta_{a+k}$ and $\tilde{\phi}_i^a$, $\tilde{p}_i^a$ are defined as before. 

In the VPAW method, the generalized eigenvalue problem (\ref{eq:H_PAW}) is solved by expanding $\tilde{\psi}$ in plane-waves. 

\begin{note}
Here we have followed the Vanderbilt scheme to generate the pseudo wave functions and the projector functions with the difference that the orthogonalized functions $p$ are taken from the Bl\"ochl construction. 
\end{note}

\subsection{Well-posedness of the VPAW method}

To be well-posed the VPAW method requires
\begin{enumerate}
\item the family of pseudo wave functions $(\tilde{\phi}_i)_{1 \leq i \leq N}$ to be independent on $[-\eta,\eta]$, so that the projector functions $(\tilde{p}_k)_{1 \leq k \leq N}$ are well defined,
\item $(\mathrm{Id}+T)$ to be invertible.
\end{enumerate}  

The conditions on the VPAW functions and parameters are given by the following propositions. Proofs can be found in Section \ref{sec:proof}. 

\begin{prop}[Linear independence of the pseudo wave functions]
\label{prop:tdphi_exist}
Let $N \in \N^*$ and $d \geq N$. There exists $\eta_0 > 0$ such that for all $0 < \eta \leq \eta_0$, the family $( \tilde{\phi}_i|_{[-\eta,\eta]})_{1 \leq i \leq N}$ is linearly independent. 
\end{prop}

\begin{prop}[Invertibility of $\mathrm{Id}+T$]
\label{cond:invrs}
The operator $\mathrm{Id}+T$ is invertible in $L^2_\mathrm{per}(0,1)$ if and only if the matrix $(\psh{\tilde{p}_k}{\phi_\ell})_{1 \leq k,\ell \leq N}$ is invertible. 
\end{prop}

From now on, we will establish our results under the following \\
\framebox[\textwidth]{ \textbf{Assumption :} the matrix $(\psh{\tilde{p}_k}{\phi_\ell})_{1 \leq k,\ell \leq N}$ is invertible for all $0 < \eta \leq \eta_0$. } \par

\section{Main results}
\label{sec:results}

We know from (\ref{eq:psi_PAW}) that 
$$
\tilde{\psi} = \psi - \sumlim{i=1}{N} (\phi_{i} - \tilde{\phi_i}) \psh{\tilde{p}_i}{\tilde{\psi}} - \sumlim{i=1}{N} (\phi_{i}^a - \tilde{\phi}_i^a) \psh{\tilde{p}_i^a}{\tilde{\psi}}.
$$
In addition, $\tilde{\psi}$ is a piecewise smooth function with first derivative jumps (due to $\psi$ and the atomic wave function $\phi_k$) at points of  $\Z$, $\Z + a$ and $d$-th derivative jumps (due to the pseudo wave functions $\tilde{\phi}_k$) at points of $\Z \pm \eta$ and $\Z + a \pm \eta$. These singularities drive the decays of the Fourier coefficients. Thus to study the Fourier convergence rate, it suffices to study the dependency of the different singularities with respect to $N$ -the number of PAW functions used-, $d$ -the smoothness of the pseudo wave functions $\tilde{\phi}_k$- and $\eta$ -the cut-off radius.

\begin{prop}[Derivative jumps at $0$]
\label{lem:der_jump_0}
Let $N \in \N^*$ and $d \geq N$. Then, there exists a positive constant $C$ independent of $\eta$ such that for $0 \leq j \leq N-1$
\begin{equation}
\label{eq:der_jump_1}
\forall \, 0 < \eta \leq \eta_0, \ \left| \left[\tilde{\psi}^{(2j+1)} \right]_0 \right| \leq C \eta^{2N-2j} \ ,
\end{equation}
and for $j \geq N$
\begin{equation}
\label{eq:der_jump_2}
\forall \, 0 < \eta \leq \eta_0, \ \left| \left[\tilde{\psi}^{(2j+1)} \right]_0 \right| \leq C .
\end{equation}
\end{prop}

The proof of Proposition \ref{lem:der_jump_0} relies on the particular structure induced by the equations satisfied by $\psi$ and $\phi_i$. Locally around a Dirac potential, their singularities have the same behavior. More precisely, if we consider the even part $\psi_e$ of $\psi$, the best approximation of $\psi_e$ by $N$ eigenfunctions $\phi_i$ is of order $2N$. It is then possible to rewrite the singularity at $0$ of $\tilde{\psi}$ to make use of this approximation. 

\begin{prop}[$d$-th derivative jump at $\eta$]
\label{lem:der_jump_eta}
Let $N \in \N^*$ and $d \geq N$. There exists a constant $C$ independent of $\eta$ such that for $d \leq k \leq 2d-2$
$$
\forall \, 0 < \eta \leq \eta_0, \ \left| \left[\tilde{\psi}^{(k)} \right]_\eta \right| \leq \frac{C}{\eta^{k-1}}.
$$
\end{prop}

The derivative jump of $\tilde{\psi}$ at $\eta$ is due to the lack of regularity of the pseudo wave functions $\tilde{\phi}_j$ at $\eta$. The latter can be written as rescaled polynomials $P_{2d-2}(\tfrac{x}{\eta})$ where $P_{2d-2}$ is of degree at most $2d-2$. If we suppose that the coefficients of $P_{2d-2}$ are uniformly bounded in $\eta$ and if the dependence on $\eta$ of the projector functions $\tilde{p}_k$ is neglected, by deriving $k$ times the polynomials $P_{2d-2}(\tfrac{x}{\eta})$, $k \geq d$, we can see why the derivative jump of $\tilde{\psi}$ at $\eta$ is expected to grow as $\eta^{-k}$. Tracking all the dependencies on $\eta$, we can in fact show that a factor $\eta$ can be gained, which is in full agreement with Figure \ref{grph:der_jump}.

Using Proposition \ref{prop:Fourier_decay} and classical estimates on eigenvalue approximations \cite{weinberger1974variational}, we have the following theorems.

\begin{theo}[Estimates on the Fourier coefficients]
\label{theo:FourierPsi}
Let $N \in \N^*$ and $d \geq N$. Let $\widehat{\tilde{\psi}}_m$ be the $m$-th Fourier coefficient of $\tilde{\psi}$. There exists a constant $C >0$ independent of $\eta$ and $m$ such that for all $0 < \eta \leq \eta_0$ and $m \geq \tfrac{1}{\eta}$
$$
\left| \widehat{\tilde{\psi}}_m \right| \leq C \left( \frac{\eta^{2N}}{m^2} + \frac{1}{\eta^{d-1}m^{d+1}} \right).
$$
\end{theo}

\begin{theo}[Estimates on the eigenvalues]
\label{theo:energies}
Let $N \in \N^*$ and $d \geq N$. Let $E_M^\eta$ be an eigenvalue of the variational approximation of \eqref{eq:H_PAW} in a basis of $M$ plane-waves and for a cut-off radius $0 < \eta \leq \eta_0$, and let $E$ be the corresponding exact eigenvalue. There exists a constant $C >0$ independent of $\eta$ and $M$ such that for all $0 < \eta \leq \eta_0$ and $M \geq \tfrac{1}{\eta}$
\begin{equation}
\label{eq:thm_eigenvalue}
0 < E_M^\eta - E \leq C \left( \frac{\eta^{4N}}{M} + \frac{1}{\eta^{2d-2}} \frac{1}{M^{2d-1}} \right).
\end{equation}
\end{theo}

The first term has the same asymptotic decay in $M$ as the brute force discretization of the problem with the original Dirac potential. 
However the prefactor $\eta^{4N}$ can be made small by using a small cut-off radius $\eta$ and/or a large $N$. Doing so, we introduce another 
error term which decays as  $M^{1-2d}$, with a prefactor of order ${\eta}^{2-2d}$. 
A natural strategy would thus be to balance these two error terms. This allows one to choose the numerical parameters in a consistent way. 
The numerical tests in Section \ref{sec:numerics} suggest that the estimate \eqref{eq:thm_eigenvalue} is optimal.

\section{Proofs}
\label{sec:proof}

This section is organized as follows. First, we prove that the VPAW method is well defined. The remainder of the section is then devoted to the proofs of Theorems \ref{theo:FourierPsi} and \ref{theo:energies}. After estimating the decay of the Fourier coefficients of the pseudo wave function $\tilde{\psi}$, we will precisely characterize the singularities of the functions $\psi$ and $\phi_k$ in order to estimate the derivative jumps of $\tilde{\psi}$. 

\subsection{Well-posedness of the VPAW method}

\begin{proof}[Proof of Proposition \ref{prop:tdphi_exist}]
To prove the linear independence of the pseudo wave functions is equivalent to show that the matrix $(\phi_{j+1}^{(k)}(\eta))_{0 \leq j \leq N-1, 0 \leq k \leq d-1}$ is full rank. In fact, we will show that the submatrix $(\phi_{j+1}^{(k)}(\eta))_{0 \leq j,k \leq N-1}$ is invertible.
Using the expression of $\phi_{j+1}$, we have for $0 \leq j,k \leq N-1$, 
$$
\phi_{j+1}^{k}(\eta) = \begin{cases}
(-1)^{k/2} \omega_{j+1}^{k} \cos(\omega_{j+1} (\eta - \tfrac{1}{2})) & \text{if } k \text{ is even}, \\
(-1)^{(k+1)/2} \omega_{j+1}^{k} \sin(\omega_{j+1} (\eta - \tfrac{1}{2})) & \text{if } k \text{ is odd}
\end{cases}
$$
Let $A(\eta)$ the matrix defined by
$$
A(\eta) := \begin{pmatrix}
\\
\cos (\omega_{j+1}(\eta - \tfrac{1}{2})) & - \omega_{j+1} \sin (\omega_{j+1} (\eta - \tfrac{1}{2})) & - \omega_{j+1}^2 \cos (\omega_{j+1} (\eta - \tfrac{1}{2})) & \cdots \\
\\
\end{pmatrix}.
$$ 
The function $\eta \mapsto \det A(\eta)$ is complex analytic, thus if it is not identically equal to $0$, there exists an interval $(0,\eta_0)$ with $\eta_0 >0$ such that the matrix $A(\eta)$ is invertible. It suffices to show that there exists $\eta \in \C$ such that $A(\eta)$ is invertible. Let $\eta = -ix + \tfrac{1}{2}$, $x >0$. Then for $x$ large, we have $\cos(\omega_{j+1} x) \sim \tfrac{1}{2} e^{\omega_{j+1} x}$ and $\sin(\omega_{j+1} x) \sim \tfrac{1}{2} e^{\omega_{j+1} x}$, thus
\begin{align*}
A(-ix + \tfrac{1}{2}) & = \frac{1}{2} \begin{pmatrix} \\ e^{\omega_{j+1} x} & -\omega_{j+1} e^{\omega_{j+1} x} & -\omega_{j+1}^2 e^{\omega_{j+1} x} \cdots  \\ \\\end{pmatrix} +\varepsilon(x) ,
\end{align*}
where 
$$
\|\varepsilon(x)\| \ll \left\| \begin{pmatrix} \\ e^{\omega_{j+1} x} & -\omega_{j+1} e^{\omega_{j+1} x} & -\omega_{j+1}^2 e^{\omega_{j+1} x} \cdots \\ \\ \end{pmatrix} \right\|.
$$
We have 
$$
\begin{pmatrix} \\ e^{\omega_{j+1} x} & -\omega_{j+1} e^{\omega_{j+1} x} & -\omega_{j+1}^2 e^{\omega_{j+1} x} \cdots \\ \\ \end{pmatrix} = \begin{pmatrix}
e^{\omega_1 x} & & 0 \\
 & \ddots & \\
 0& & e^{\omega_{N} x}
\end{pmatrix} \begin{pmatrix}
1 & \omega_1 & & \omega_{1}^{N-1} \\
\vdots & \vdots &  \cdots & \vdots \\
1 & \omega_{N} & & \omega_{N}^{N-1}
\end{pmatrix}
$$
which is invertible because the phases $(\omega_j)_{1 \leq j \leq N}$ are pairwise distinct. Hence $A(-ix + \tfrac{1}{2})$ is invertible for $x$ large enough. 
\end{proof}

\begin{proof}[Proof of Proposition \ref{cond:invrs}]
As $T$ is a finite rank and thus compact operator, proving the statement is equivalent to show that $\mathrm{Ker}(\mathrm{Id}+T) = \{0\}$. First, suppose that the matrix $(\psh{\tilde{p}_k}{\phi_i})_{k,i}$ is invertible and let $f \in \mathrm{Ker}(\mathrm{Id}+T)$. We have
\begin{equation}
\sum_{i=1}^{N} \psh{\tilde{p}_i}{f} (\phi_{i} - \tilde{\phi_i}) + f = 0
\label{eq:pfcond}
\end{equation}
Since $\phi_i - \tilde{\phi_i}$ is supported in $[-\eta, \eta]$ we also have $\mathrm{supp}(f) \subset [-\eta, \eta]$.

By multiplying each side of Equation (\ref{eq:pfcond}) by $\tilde{p}_k$, $1 \leq k \leq N$ and integrating on $[-\frac{1}{2},\frac{1}{2}]$, we obtain:
$$
0  = \sum_{i=1}^{N} \psh{\tilde{p}_i}{f} \psh{\tilde{p}_k}{\phi_{i} - \tilde{\phi_i}} + \psh{\tilde{p}_k}{f} = \sum_{i=1}^{N} \psh{\tilde{p}_i}{f} ( \psh{\tilde{p}_k}{\phi_i} - \underbrace{\psh{\tilde{p}_k}{\tilde{\phi_i}}}_{=\delta_{ki}}) + \psh{\tilde{p}_k}{f} \ , 
$$
so that
$$
\forall 1 \leq k \leq N, \ 0 = \sum_{i=1}^{N} \psh{\tilde{p}_i}{f} \psh{\tilde{p}_k}{\phi_i}.
 $$
Since we assumed that the matrix $(\psh{\tilde{p}_k}{\phi_i})_{k,i}$ is invertible,
$$
\forall \, 1 \leq i \leq N, \ \psh{\tilde{p}_i}{f} = 0 \ .
$$
Going back to (\ref{eq:pfcond}), this implies $f=0$ and $\mathrm{Id}+T$ is invertible.

Now we suppose that the matrix $(\psh{\tilde{p}_k}{\phi_i})_{k,i}$ is not invertible. Thus there is $(\alpha_i)_{1 \leq i \leq N}$ such that
$$
\forall 1 \leq i \leq N, \ \sumlim{j=1}{N} \alpha_j \psh{\tilde{p}_i}{\phi_j} = 0.
$$
Let $f(x) = \sumlim{j=1}{N} \alpha_j (\phi_j - \tilde{\phi}_j)$. Then
\begin{align*}
(\mathrm{Id}+T)f & = \sumlim{j=1}{N} \alpha_j (\phi_j - \tilde{\phi}_j) + \sumlim{i=1}{N} \sumlim{j=1}{N} \alpha_j \psh{\tilde{p}_i}{\phi_j - \tilde{\phi}_j}(\phi_j - \tilde{\phi}_j) \\
& = \sumlim{i=1}{N} \sumlim{j=1}{N} \alpha_j \psh{\tilde{p}_i}{\phi_j}(\phi_j - \tilde{\phi}_j) \\
& = 0
\end{align*} 
Thus $\mathrm{Ker}(I+T) \not= \lbrace 0 \rbrace$ and $(I+T)$ is not invertible.
\end{proof}

\subsection{Structure and approximation lemmas}

A key intermediate result in our study is the estimation of the decay of the Fourier coefficients of $\tilde{\psi}$ as a function of its derivative jumps.

\begin{prop}
\label{prop:Fourier_decay}
Let $N \in \N^*$ and $d \geq N$. Let $\widehat{\tilde{\psi}}_m$ be the $m$-th Fourier coefficient of $\tilde{\psi}$. Then
%\begin{multline}
%\widehat{\tilde{\psi}}_m = \sumlim{j=0}{\lfloor \frac{d}{2} \rfloor -1} \frac{1}{(2i\pi m)^{2+2j}} \left[\tilde{\psi}^{(2j+1)}\right]_0 + \frac{e^{-2i \pi m \eta}}{(2i\pi m)^{d+1}} \left[\tilde{\psi}^{(d)} \right]_\eta  + \frac{e^{2i \pi m \eta}}{(2i\pi m)^{d+1}} \left[\tilde{\psi}^{(d)}\right]_{-\eta}  \\
%+ e^{-2i \pi m a} \left( \sumlim{j=0}{\lfloor \frac{d}{2} \rfloor -1} \frac{1}{(2i\pi m)^{2+2j}} \left[\tilde{\psi}^{(2j+1)}\right]_a + \frac{e^{-2i \pi m \eta}}{(2i\pi m)^{d+1}} \left[\tilde{\psi}^{(d)}\right]_{a+\eta}  + \frac{e^{2i \pi m \eta}}{(2i\pi m)^{d+1}} \left[\tilde{\psi}^{(d)}\right]_{a-\eta}  \right) \\
%+ \cO\left( \frac{1}{m^{d+2}} \right)
%\end{multline}
\begin{equation*}
\begin{split}
\widehat{\tilde{\psi}}_m =& \sumlim{j=0}{\lfloor \frac{d}{2} \rfloor -1} \frac{1}{(2i\pi m)^{2+2j}} \left[\tilde{\psi}^{(2j+1)}\right]_0 + \sumlim{k=d}{2d-2} \frac{ e^{\mp 2i \pi m \eta}}{(2i\pi m)^{k+1}} \left[\tilde{\psi}^{(k)} \right]_{\pm\eta}   \\
 &+ e^{-2i \pi m a} \left( \sumlim{j=0}{\lfloor \frac{d}{2} \rfloor -1} \frac{1}{(2i\pi m)^{2+2j}} \left[\tilde{\psi}^{(2j+1)}\right]_a + \sumlim{k=d}{2d-2} \frac{ e^{\mp 2i \pi m \eta}}{(2i\pi m)^{k+1}} \left[\tilde{\psi}^{(k)} \right]_{a\pm\eta}  \right) \\
 &+ \frac{1}{(2i\pi m)^{2d-1}} \itg{0}{1}{\tilde{\psi}^{(2d-1)}(x)e^{-2i\pi mx}}{x}.
\end{split}
\end{equation*}
\end{prop}

\begin{proof}
This result follows from the definition of the Fourier coefficients and integration by parts. 
\end{proof}

In view of the Proposition \ref{prop:Fourier_decay}, the decay of the Fourier coefficients can be inferred from the derivative jumps of $\tilde{\psi}$ according to the VPAW parameters. The singularities of $\tilde{\psi}$ at integer values are caused by the singularity of the functions $\psi$ and $\phi_k$. Thus, to get an accurate characterization of the singularities of $\tilde{\psi}$, we need to precisely know how the functions $\psi$ and $\phi_k$ behave in a neighborhood of their singularities.

\begin{lem}[Structure lemma]
\label{lem:decomp}
Let $\psi$ be an eigenfunction (\ref{eq:H_mol}) associated to the eigenvalue $E$. Then in a neighborhood of $0$, we have the following expansion :
\begin{multline}
\psi(x) = \psi(0) \left( \sumlim{j=0}{k} \frac{(-E)^j}{(2j)!} x^{2j} - \frac{Z_0}{2} \sumlim{j=0}{k} \frac{(-E)^j}{(2j+1)!} |x|^{2j+1} \right) \\
+ \frac{\psi'(0_+) + \psi'(0_-)}{2} \sumlim{j=0}{k} \frac{(-E)^j}{(2j+1)!} x^{2j+1} + \psi_{2k+2}(x),
\end{multline}
where $\psi_{2k+2}$ is a $C^{2k+2}$ function satisfying in a neighbourhood of $0$,
$$
\begin{cases}
\psi_{2k+2}^{(2k+2)} = (-E)^{k+1} \psi , \\
|\psi_{2k+2}(x)| \leq C \frac{|-E|^{k+1}}{(2k+2)!} |x|^{2k+2}.
\end{cases}
$$
\end{lem}

\begin{proof}
This lemma is proved by induction.  

\paragraph*{Initialization}
For $k = 0$, let 
$$
\theta_2(x) = \psi(x) + Z_0 \frac{|x|}{2} \psi(0) .
$$
We differentiate $\theta_2$ twice:
\begin{equation}
\label{eq:theta2}
\theta_2''(x) = \psi''(x) - [\psi']_0 \delta_0  = -E \psi(x) , \quad \text{on } (-\tfrac{1}{2},\tfrac{1}{2}).
\end{equation}
The function $\psi$ being continuous, $\theta_2$ is $C^2$ in a neighborhood of $0$. Moreover,
$$
\theta_2(0) = \psi(0),
$$
and
$$
\theta_2'(x) = \psi'(x) - \frac{\mathrm{sign}(x)}{2}[\psi']_0 .
$$
When $x$ tends to $0_+$ or $0_-$, we obtain the same expression: 
$$
\theta_2'(0) = \frac{\psi'(0_+) + \psi'(0_-)}{2}.
$$
Setting
$$
\psi_2(x) = \theta_2(x) - \psi(0) - \frac{\psi'(0_+) + \psi'(0_-)}{2} x ,
$$
the statement is true for $k=0$.

\paragraph{Inductive step}
Suppose the statement is true for $k-1$. Then, we have in a neighbourhood of 0,
$$
\psi_{2k}^{(2k)}(x) = (-E)^{k} \psi(x) .
$$
Let
$$
\theta_{2k+2}(x) = \psi_{2k}(x) - (-E)^{k} \frac{|x|}{2} [\psi']_0 \frac{x^{2k}}{(2k+1)!} .
$$
Then
$$
\theta_{2k+2}^{(2k)}(x) = (-E)^{k} \left( \psi(x) - \frac{|x|}{2} [\psi']_0 \right) ,
$$
so that in view of \eqref{eq:theta2},
\begin{align*}
\theta_{2k+2}^{(2k+2)}(x) & = (-E)^k \left(\psi''(x) - [ \psi' ]_0 \delta_0 \right) = (-E)^{k+1} \psi(x),
\end{align*}
in the neighbourhood of 0.

So $\theta_{2k+2}$ is a $C^{2k+2}$ function in a neighbourhood of 0 and we have:
\begin{multline*}
\psi(x) = \psi(0) \left( \sumlim{j=0}{k-1} \frac{(-E)^j}{(2j)!} x^{2j} - Z_0 \sumlim{j=0}{k} \frac{(-E)^j}{(2j+1)!} \frac{|x|^{2j+1}}{2} \right) \\
+ \sumlim{j=0}{k-1} \frac{(-E)^j}{(2j+1)!} \frac{\psi'(0_+) + \psi'(0_-)}{2} x^{2j+1} + \theta_{2k+2}(x)  .
\end{multline*}
Il suffices to evaluate $\theta_{2k+2}^{(2k)}(0)$ and $\theta_{2k+2}^{(2k+1)}(0)$ to conclude the proof. We have
\begin{equation*}
\theta_{2k+2}^{(2k)}(0) = (-E)^{k} \psi(0) \ ,
\end{equation*}
and
\begin{equation*}
\theta_{2k+2}^{(2k+1)}(x) = (-E)^{k} \left( \psi'(x) - \frac{\mathrm{sign}(x)}{2} [\psi']_0 \right) \ ,
\end{equation*}
so if $x$ tends to $0$, we have :
\begin{equation*}
\theta_{2k+2}^{(2k+1)}(0) = \frac{\psi'(0_+) + \psi'(0_-)}{2} (-E)^k .
\end{equation*}
Define
\begin{equation*}
\psi_{2k+2}(x) = \theta_{2k+2}(x) - \psi(0) \frac{x^{2k}}{(2k)!}  - \frac{\psi'(0_+) + \psi'(0_-)}{2} \frac{x^{2k+1}}{(2k+1)!} \ ,
\end{equation*}
and the induction is proved. 
\end{proof}

Let $\psi_e$ be the even part of $\psi$. We have in a neighbourhood of $0$
\begin{equation}
\label{eq:psi_e}
\psi_e(x) = \psi(0) \sumlim{k=0}{N-1} \left( \frac{x^{2k}}{(2k)!}  - \frac{Z_0}{2} \frac{|x|^{2k+1}}{(2k+1)!} \right) (-E)^k + \frac{1}{2}(\psi_{2N}(x) + \psi_{2N}(-x)).
\end{equation}

\begin{lem}[Approximation]
\label{lem:approx}
There exist constants $(c_j)_{1 \leq j \leq N} \in \mathbb{R}^N$ satisfying 
\begin{equation*}
\psi_e(x) = \sumlim{j=1}{N} c_j \phi_j(x) + \cO\left(x^{2N}\right), \qquad \text{as } x \to 0.
\end{equation*}
\end{lem}

\begin{proof}
By Lemma \ref{lem:decomp} applied to $\phi_j$, we have
\begin{equation*}
\phi_j(x) = \phi_j(0) \left( \sumlim{k=0}{N-1} \frac{(-\epsilon_j)^k}{(2k)!} x^{2k} - \frac{Z_0}{2} \sumlim{k=0}{N-1} \frac{(-\epsilon_j)^k}{(2k+1)!} |x|^{2k+1} \right) + \phi_{j,2N}(x) \ ,
\end{equation*}
where
\begin{equation*}
\phi_{j,2N}(x) = \cO(x^{2N}).
\end{equation*}
So
\begin{multline*}
\psi_e(x) - \sumlim{j=1}{N} c_j \phi_j(x) = \sumlim{k=0}{N-1} \left( \psi(0) (-E)^k - \sumlim{j=1}{N} (-\epsilon_j)^k c_j \phi_j(0) \right) \frac{x^{2k}}{(2k)!} \\
- \frac{Z_0}{2} \sumlim{k=0}{N-1} \left( \psi(0) (-E)^k - \sumlim{j=1}{N} (-\epsilon_j)^k c_j \phi_j(0) \right) \frac{|x|^{2k+1}}{(2k+1)!} + \psi_{2N}(x) - \sumlim{j=1}{N} c_j \phi_{j,2N}(x).
\end{multline*}
To prove the lemma, it remains to show that there exist coefficients $(c_j)$ such that : 
\begin{equation*}
\forall \, 0 \leq k \leq N-1, \,  \sumlim{j=1}{N} (-\epsilon_j)^k c_j \phi_j(0) = \psi(0) (-E)^k \ .
\end{equation*}
We have chosen the functions $\phi_j$ so that $\phi_j(0) \not= 0$. By defining $\alpha_j = c_j \phi_j(0)$, we recognize a Vandermonde linear system. The eigenvalues $\epsilon_j$ are all different so the system is invertible and the lemma is proved. 
\end{proof}

\subsection{Derivative jumps at $0$}

Recall $p_i(t) = \rho_\eta(t) \tilde{\phi}_i(t)$. We will introduce some notation used in the next proofs:
\begin{align*}
p(t) & := (p_1(t), \dots, p_{N}(t))^T \in \R^N, \\
\psh{\tilde{p}}{\tilde{\psi}} & := (\psh{\tilde{p}_1}{\tilde{\psi}}, \dots, \psh{\tilde{p}_{N}}{\tilde{\psi}})^T \in \R^N, \\
\psh{\tilde{p}}{\psi} & := (\psh{\tilde{p}_1}{\psi}, \dots, \psh{\tilde{p}_{N}}{\psi})^T \in \R^N, \\
\Phi(t) & := (\phi_1(t), \dots, \phi_{N}(t))^T \in \R^N, \\
\widetilde{\Phi}(t) & := (\tilde{\phi}_1(t), \dots, \tilde{\phi}_{N}(t))^T \in \R^N, \\
B & := ( \langle p_i, \tilde{\phi_j} \rangle)_{1 \leq i,j \leq N} \in \R^{N \times N}, \\
\widetilde{A} & :=  ( \langle \tilde{p}_i, \phi_j \rangle)_{1 \leq i,j \leq N} \in \R^{N \times N}, \\
A & := ( \langle p_i, \phi_j \rangle)_{1 \leq i,j \leq N} = B\widetilde{A} \in \R^{N \times N}.
\end{align*}

\begin{lem}
\label{lem:der_jump}
Let $(c_k)_{1 \leq k \leq N}$ be any vector of $\mathbb{R}^N$ and $j \in \mathbb{N}$. Then,
\begin{equation*}
[\tilde{\psi}^{(2j+1)}]_0 = -Z_0 \left( (-E)^j \psi(0) - \sumlim{k=1}{N} c_k (-\epsilon_k)^j \phi_k(0) - \psh{A^{-1}p}{\psi - \sumlim{k=1}{N} c_k \phi_k} \cdot \cE^j \Phi(0) \right) .
\end{equation*}
where $\cE$ is the diagonal matrix 
\begin{equation*}
\cE =
\begin{pmatrix}
- \epsilon_1 &  0 & \ldots & 0 \\
0 & - \epsilon_2 & \ldots & \vdots \\
\vdots & \vdots & \ddots & 0 \\
0  &   \ldots  & 0  & - \epsilon_{N}
\end{pmatrix}.
\end{equation*}
\end{lem}

\begin{proof}
We first prove the statement for $j=0$. In a neighbourhood of $0$, we have:
\begin{equation}
\tilde{\psi} = \psi - \sumlim{i=1}{N} \psh{\tilde{p}_i}{\tilde{\psi}} (\phi_i - \tilde{\phi}_i).
\label{eq:psi_tilde-psi1}
\end{equation}
Using the equations satisfied by $\psi$ and $\phi_i$ gives for the first derivative jump at $0$ of $\tilde{\psi}$ :
\begin{align}
[\tilde{\psi}']_0 & = [\psi']_0 - \sumlim{i=1}{N} \psh{\tilde{p}_i}{\tilde{\psi}} [\phi_i']_0 \nonumber \\
& = -Z_0 \left( \psi(0) - \sumlim{i=1}{N} \psh{\tilde{p}_i}{\tilde{\psi}} \phi_i(0) \right).
\label{eq:saut_tdpsi}
\end{align}
Multiplying equation (\ref{eq:psi_tilde-psi1}) by $\tilde{p}_k$ and integrating on $[-1/2,1/2]$,
\begin{align*}
\psh{\tilde{p}_k}{\psi} - \sumlim{i=1}{N} \psh{\tilde{p}_i}{\tilde{\psi}} \psh{\tilde{p}_k}{\phi_i} = & \psh{\tilde{p}_k}{\tilde{\psi}} - \sumlim{i=1}{N} \psh{\tilde{p}_i}{\tilde{\psi}} \underbrace{\psh{\tilde{p}_k}{\tilde{\phi_i}}}_{=\delta_{ki}} =0 \ , \\
 \psh{\tilde{p}_k}{\psi} = & \sumlim{i=1}{N} \psh{\tilde{p}_i}{\tilde{\psi}} \psh{\tilde{p}_k}{\phi_i} \ .
\end{align*}
Therefore
\begin{equation}
\label{eq:p_psi-p_tildepsi}
\widetilde{A} \psh{\tilde{p}}{\tilde{\psi}} = \psh{\tilde{p}}{\psi}
\end{equation}
Likewise
$$
B \psh{\tilde{p}}{\psi} = \psh{p}{\psi},
$$
since for $1 \leq i \leq N$ :
\begin{align*}
\sumlim{j=1}{N} B_{ij} \psh{\tilde{p}_j}{\psi} & = \sumlim{j=1}{N} \psh{p_i}{\tilde{\phi}_j} \psh{\tilde{p}_j}{\psi} \\
& = \psh{p_i}{\psi}
\end{align*}
According to Proposition \ref{cond:invrs}, the matrix $\widetilde{A}$ is invertible and so is $B$ since $B = A \widetilde{A}$, with $A$ invertible by assumption. We therefore have
\begin{align*}
\sumlim{i=1}{N} \psh{\tilde{p}_i}{\tilde{\psi}} \phi_i(0) & = \psh{\tilde{p}}{\tilde{\psi}} \cdot \Phi(0) \\
& = \widetilde{A}^{-1} \psh{\tilde{p}}{\psi} \cdot \Phi(0)\\
& = (B \widetilde{A})^{-1} \psh{p}{\psi} \cdot \Phi(0) \\
& = \psh{A^{-1} p}{\psi} \cdot \Phi(0) .
\end{align*}
We thus obtain the more compact form:
\begin{equation*}
[\tilde{\psi}']_0 =  -Z_0 \left( \psi(0) - \psh{A^{-1} p}{\psi} \cdot \Phi(0) \right) .
\end{equation*}
To complete the proof the lemma, it suffices to show that 
\begin{equation*}
\psh{A^{-1} p}{\phi_{i+1}} = e_i ,
\end{equation*}
where $e_i$ is the $i$-th vector of the canonical basis of $\mathbb{R}^N$. This is straightforward since $\psh{p}{\phi_{i+1}}$ is simply the $i$-th column of the matrix $A$.

For $j \geq 1$, we proceed in the same way using 
\begin{equation*}
\begin{cases}
[\psi^{(2j+1)}]_0 =  -Z_0 (-E)^j \psi(0) , \\
[\phi_k^{(2j+1)}]_0 = -Z_0 (-\epsilon_k)^j \phi_k(0).
\end{cases}
\end{equation*}
\end{proof}

\begin{note}
Notice that we showed
\begin{equation}
\psh{\tilde{p}}{\tilde{\psi}} = \psh{A^{-1}p}{\psi}.
\label{eq:tilde_p-Ap-psi}
\end{equation}
This equality will be used later in the estimation of the $d$-th derivative jump.
\end{note}

To prove Lemma \ref{lem:der_jump_0}, it remains to study the behavior of $A^{-T} \cE^j \Phi(0)$ as $\eta$ goes to $0$. By assumption, $A$ is invertible for all $\eta >0$ but when $\eta = 0$, $A$ is a rank $1$ matrix. Actually, in the special case of the 1D Schrödinger operator with Dirac potentials, we have a precise characterization of the behavior of $A^{-T} \cE^j \Phi(0)$ as $\eta$ goes to $0$. 

\begin{lem}
\label{lem:ApPhi}
Let $f$ be a function in $L^2_\mathrm{per}(0,1)$ and $Q(t) = (Q_0(t), \dots, Q_{d-1}(t))^T$ be a vector of even polynomials which forms a basis of the space of even polynomials of degree at most $2d-2$. 
Let $G_\eta$ be the $d \times N$ matrix and $C_\eta$ the $N \times d$ matrix defined by:
\begin{align*}
G_\eta & = \itg{-1}{1}{\rho(t) Q(t) \Phi(\eta t)^T}{t} , \\
\widetilde{\Phi}(t) & = C_\eta Q(t/ \eta), \ \forall t \in (-\eta, \eta).
\end{align*}
Then we have
\begin{equation}
\psh{A^{-1} p}{f} = \itg{-1}{1}{\rho(t) (C_\eta G_\eta)^{-1} C_\eta Q(t) f(\eta t)}{t} .
\label{eq:ApPhi}
\end{equation}
\end{lem}

\begin{proof}
We have
\begin{align*}
A_{ij} & = \psh{p_{i+1}}{\phi_{j+1}} \\
& = \itg{-\eta}{\eta}{\rho(t/\eta) \tilde{\phi}_{i+1}(t) \phi_{j+1}(t)}{t} \\
& = \eta \itg{-1}{1}{\rho(t) \tilde{\phi}_{i+1}(\eta t) \phi_{j+1}(\eta t)}{t} . 
\end{align*}
Therefore,
\begin{align*}
A & = \eta \itg{-1}{1}{\rho(t) \widetilde{\Phi}(\eta t) \Phi(\eta t)^T}{t}  \\
& = \eta \, C_\eta \itg{-1}{1}{\rho(t) Q(t) \Phi(\eta t)^T}{t} \\
& = \eta \, C_\eta G_\eta .
\end{align*}
Since $(\tilde{\phi}_i)_{1 \leq i \leq N}$ is free, the matrix $C_\eta$ is invertible :
\begin{equation*}
A^{-1} p = \frac{1}{\eta} \rho(t/\eta) (C_\eta G_\eta)^{-1} C_\eta Q(t/\eta) .
\end{equation*}
Thus
\begin{align*}
\psh{A^{-1} p}{f} & = \frac{1}{\eta} \itg{-\eta}{\eta}{\rho(t/\eta) (C_\eta G_\eta)^{-1} C_\eta Q(t/\eta) f( t)}{t}  \\
& = \itg{-1}{1}{\rho(t)(C_\eta G_\eta)^{-1} C_\eta Q(t) f(\eta t)}{t} .
\end{align*}
\end{proof}

Before moving to the next lemma, we introduce the following notation. Let $Q_k$ be the even polynomials of degrees at most $2d-2$ defined by
\begin{equation*}
\itg{-1}{1}{\rho(t) Q_k(t) t^{2j}}{t} = \delta_{kj}, \ 0 \leq k,j \leq d-1 \ ,
\end{equation*}
and let $P_k$, $0 \leq k \leq d-1$ and $P$ be defined by
\begin{align*}
P_k(t) & = \frac{1}{2^k k!}(t^2-1)^k \ , \\
P(t)&  = (P_0(t), \dots, P_{d-1}(t))^T .
\end{align*}
It is easy to see that $P_k$ satisfies
\begin{equation*}
\begin{cases}
P_k^{(j)}(1) = 0 , \ 0 \leq j \leq k-1 \ , \\
P_k^{(k)}(1) = 1.
\end{cases}
\end{equation*}
Let $\Pi$ be the transition matrix from $Q$ to $P$ :
\begin{equation*}
\Pi Q = P.
\end{equation*}
Finally we denote by $C_\eta$ the matrix of the expansion of $\widetilde{\Phi}$ in the basis $Q$ and $C_\eta^{(P)}$ the matrix of the expansion of $\widetilde{\Phi}$ in the basis $P$ :
\begin{equation*}
\begin{cases}
\widetilde{\Phi}(t) = C_\eta Q(t/\eta) \ , \\
\widetilde{\Phi}(t) = C_\eta^{(P)} P(t/\eta).
\end{cases}
\end{equation*}
It is easy to see that 
\begin{equation*}
C_\eta = C_\eta^{(P)} \Pi.
\end{equation*}

\begin{lem}
\label{lem:stab}
For $0 \leq j \leq N-1$, we have
\begin{equation}
\label{eq:stability}
\frac{\eta^{2j}}{(2j)!} (\cE^j \Phi(0))^T (C_\eta G_\eta)^{-1} C_\eta \ \mathop{\longrightarrow}_{\eta \to 0} \ e_j^T \begin{pmatrix} I_N & M_\pi \end{pmatrix},
\end{equation}
where $M_\pi$ is a $(d-N) \times N$ matrix. 

Furthermore
\begin{equation*}
\| (C_\eta G_\eta)^{-1} C_\eta \| = \cO \left( \frac{1}{\eta^{2N-2}} \right).
\end{equation*}
\end{lem}

\begin{note}
The main idea of the proof is to use the particular structures of the matrices $C_\eta$ and $G_\eta$. We denote by $C_1$, $C_2$, $G_1$ and $G_2$ the matrices defined by
\begin{equation*}
\begin{cases}
C_\eta =  \Big( C_1 \, \Big| \, C_2 \Big) , \\
G_\eta = 
\left( \begin{array}{c}
G_1 \\ \hline  G_2 
\end{array} \right).
\end{cases}
\end{equation*}

Suppose that $C_1$ is invertible and such that $\| C_1^{-1}C_2 \| = \cO (1)$ as $\eta \to 0$, and there exists an invertible matrix $H_1$ such that 
\begin{equation*}
\begin{cases}
G_1 H_1 = I_N + \cO(\eta) \ , \\
G_2 H_1 = \cO(\eta) \ , \\
e_0^T H_1^{-1} = \Phi(0)^T.
\end{cases}
\end{equation*}

Then it is easy to see that
\begin{equation*}
\Phi(0)^T (C_\eta G_\eta)^{-1} C_\eta = e_0^T \Big( I_N \ \Big| \ C_1^{-1}C_2 \Big) + \cO(\eta).
\end{equation*}

Using Lemma \ref{lem:decomp} applied to $\Phi$, it is easy to unveil the dependence in $\eta$ of the matrix $G_\eta$ but we have no hint on the structure of $C_\eta$. Likewise, $C_\eta^{(P)}$ is easy to study but the matrix $\itg{-1}{1}{\rho(t) P(t) \Phi(\eta t)^T}{t}$ is not. So we have to work with both bases $P$ and $Q$, exhibit the structures of the matrices $C_\eta^{(P)}$ and $G_\eta$ and recombine everything with the transition matrix $\Pi$. 
\end{note}

Before proving Lemma \ref{lem:stab}, we state some properties of the matrix $C_\eta^{(P)}$ and its submatrices $C_1$, $C_2$.

\begin{lem}
\label{lem:C_eta}
Let $N \in \N^*$ and $d \geq N$. Let $C_1 \in \R^{N \times N}$ and $C_2 \in \R^{N \times (d-N)}$ be the matrices such that:
\begin{equation}
\label{eq:C_etaP}
C_\eta^{(P)}  = \Big( C_1 \, \Big| \, C_2 \Big).
\end{equation}
Let $(g_k)_{0 \leq k \leq N-1}$ be the dual family of the vectors $(\eta^k \Phi^{(k)}(\eta))_{0 \leq k \leq N-1}$ and $K_1$ be the matrix
\begin{equation*}
K_1 = 
\begin{pmatrix}
g_0^T \\
\vdots \\
g_{N-1}^T
\end{pmatrix} \in \R^{N \times N}.
\end{equation*}
Then, there exists an upper triangular matrix $\mathcal{P}$ independent of $\eta$ of the form
\begin{equation*}
\mathcal{P} = 
\begin{pmatrix}
1 & 0 & \dots & 0 \\
0 & \ddots & * & * \\
\vdots & 0 & \ddots & *  \\
0 & \dots & 0 & 1
\end{pmatrix} \in \R^{N \times N}
\end{equation*}
such that 
\begin{equation*}
\begin{cases}
C_1^{-1} = \mathcal{P} K_1, \\
C_1^{-1} C_2 = M + \cO(\eta),
\end{cases}
\end{equation*}
where $M \in \R^{N \times (d-N)}$ is a matrix independent of $\eta$. 
\end{lem}

\begin{note}
The particular form of the matrix $\mathcal{P}$ will be used in the estimation of the $d$-th derivative jump of $\tilde{\psi}$ at $\eta$ (Lemma \ref{lem:der_jump_eta}) and of $Tf(x)$ (Lemmas \ref{lem:Tf_L2} and \ref{lem:Tf_prime}).
\end{note}

\begin{proof}
Let $(c_k)_{0 \leq k \leq d-1}$ be the columns of $C_\eta^{(P)}$. By the continuity conditions at $\eta$ and our choice of the polynomials $P_k$, we have
\begin{equation}
\forall 0 \leq j \leq d-1, \  c_j = \eta^j \Phi^{(j)}(\eta) - \sumlim{k=0}{j-1} P_k^{(j)}(1) c_k.
\label{eq:columns_Ceta}
\end{equation}
Thus $c_j$ is a linear combination of the vectors $\eta^k \Phi^{(k)}(\eta)$ for $k \leq j$ whose coefficients are \emph{independent} of $\eta$. Moreover as $P_0=1$, we have $P_0^{(j)}=0$ for $j\geq 1$. So in fact, for $j \geq 1$, $c_j$ is spanned by the vectors $\eta^k \Phi^{(k)}(\eta)$ for $1 \leq k \leq j$. Then, by definition of $K_1$ and the vectors $g_k$,
\begin{equation}
\label{eq:inv_C_etaP}
K_1 C_1 = 
\begin{pmatrix}
1 & 0 & \dots & 0 \\
0 & \ddots & * & * \\
\vdots & 0 & \ddots & *  \\
0 & \dots & 0 & 1
\end{pmatrix}.
\end{equation}
Note that this matrix is independent of $\eta$. Let us denote it by $\mathcal{P}^{-1}$. Then the inverse of $C_1$ is $\mathcal{P} K_1$ and $\mathcal{P}$ has the same structure as $\mathcal{P}^{-1}$.
Recall that for $k \geq N$, $c_k$ is a linear combination of the vectors $\eta^j \Phi^{(j)}(\eta)$ for $j \leq k$. As $g_j$ is the dual family of the vectors $(\eta^l \Phi^{(l)}(\eta))_{0 \leq l \leq N-1}$ and $\|g_j\| = \cO( \eta^{1-N} )$, we have 
\begin{equation}
\label{eq:M}
C^{-1}_1 C_2 = M + \cO(\eta),
\end{equation}
where $M$ is a $N \times (d-N)$ matrix independent of $\eta$.
\end{proof}

\begin{proof}[Proof of Lemma \ref{lem:stab}]
Let $G_1 \in \R^{N \times N}$ and $G_2 \in \R^{(d-N) \times N}$ be the unique matrices such that
\begin{equation*}
G_\eta = 
\left( \begin{array}{c}
G_1 \\ \hline  G_2 
\end{array} \right) \in \R^{d \times N}.
\end{equation*}
By definition
\begin{equation*}
G_\eta = \itg{-1}{1}{\rho(t) Q(t) \Phi(\eta t)^T}{t}.
\end{equation*}
By Lemma \ref{lem:decomp} applied to each $\phi_j$,
\begin{equation*}
\Phi(\eta t) = \sumlim{k=0}{N-1} \left( t^{2j} - \eta \frac{Z_0}{2} \frac{|t|^{2j+1}}{2j+1} \right) \frac{\eta^{2j}}{(2j)!} \cE^j \Phi(0) + \cO(\eta^{2N}).
\end{equation*}
Let $a_j \in \R^d, a_j^N \in \R^N, a_j^{d-N} \in \R^{d-N}$ be defined by
\begin{equation*}
a_j := - \itg{-1}{1}{\rho(t) \frac{Z_0}{2} \frac{|t|^{2j+1}}{2j+1} Q(t)}{t} =: \begin{pmatrix}
a_j^N \\ a_j^{d-N} \end{pmatrix}.
\end{equation*}
Then by definition of the polynomials $Q$ we have
\begin{equation*}
\begin{cases}
G_1 = \sumlim{j=0}{N-1} (e_j + \eta a_j^N) \frac{\eta^{2j}}{(2j)!} (\cE^j \Phi(0))^T + \cO(\eta^{2N}), \\
G_2 = \sumlim{j=0}{N-1} \eta a_j^{d-N} \frac{\eta^{2j}}{(2j)!} (\cE^j \Phi(0))^T + \cO(\eta^{2N}).
\end{cases}
\end{equation*}

Let $(f_0,\dots,f_{N-1})$ be the dual basis of $\left(\cE^j \Phi(0) \frac{\eta^{2j}}{(2j)!} \right)_{0 \leq j \leq N-1}$ in $\mathbb{R}^N$ and $H_1$ be the matrix
\begin{equation*}
H_1 :=
\begin{pmatrix}
f_0 & \cdots & f_{N-1}
\end{pmatrix} \in \R^{N \times N}.
\end{equation*}
It is straightforward to see that $\| H_1 \| = \cO( \eta^{2-2N})$ and 
\begin{equation*}
\begin{cases}
G_1 H_1 = I_N + \cO(\eta) \ , \\
G_2 H_1 = \cO(\eta).
\end{cases}
\end{equation*}
The matrix $H_1$ is invertible and its inverse is 
\begin{equation*}
H_1^{-1} = \sumlim{j=0}{N-1} e_j \frac{\eta^{2j}}{(2j)!} (\cE^j \Phi(0))^T.
\end{equation*}

Let us now prove \eqref{eq:stability} for $j=0$. We have
\begin{align*}
\Phi(0)^T (C_\eta G_\eta)^{-1} C_\eta & = e_0^T H_1^{-1} \left(C^{(P)}_\eta \Pi G_\eta \right)^{-1} C^{(P)}_\eta \Pi  \\
& = e_0^T \left( C^{-1}_1 C^{(P)}_\eta \Pi G_\eta H_1 \right)^{-1} \Pi \Big( I_N \ \Big| \ C^{-1}_1 C_2 \Big) \\
& = e_0^T \left(  \Big( I_N \ \Big| \ M + \cO(\eta) \Big)   \Pi \Big( I_N \ \Big| \ \cO(\eta) \Big) ^T \right)^{-1}   \Big( I_N \ \Big| \ M + \cO(\eta) \Big)  \Pi.
\end{align*}
Decomposing $\Pi$ into four blocks
\begin{equation*}
\Pi = \begin{pmatrix}
\Pi_1 & \Pi_2 \\
\Pi_3 & \Pi_4
\end{pmatrix}, \ \text{with } \Pi_1 \in \R^{N \times N},
\end{equation*}
we obtain
\begin{align*}
\Phi(0)^T (C_\eta G_\eta)^{-1} C_\eta & = e_0^T \left(\Pi_1 + M \Pi_3 + \cO(\eta) \right)^{-1} \Big( \Pi_1 + M \Pi_3 + \cO(\eta) \ \Big| \ \Pi_2 + M\Pi_4 + \cO(\eta) \Big) \\
& = e_0^T \Big( I_N \ \Big| \ M_\pi \Big) + \cO(\eta).
\end{align*}

For $1 \leq j \leq N-1$ we proceed in the same way, using $e_j^T H^{-1}_1 = \frac{\eta^{2j}}{(2j)!} (\cE^j \Phi(0))^T$.
\end{proof}

\begin{proof}[Proof of Proposition \ref{lem:der_jump_0}]
Let $0 \leq j \leq N-1$ and let $(c_k)_{1 \leq k \leq N}$ be as in Lemma \ref{lem:approx}. Then by Lemma \ref{lem:der_jump} we have : 
\begin{align*}
[\tilde{\psi}^{(2j+1)}]_0  & =  \underbrace{(-E)^j \psi(0) - \sumlim{k=1}{N} c_k (-\epsilon_k)^j \phi_k(0)}_{=0} + \psh{A^{-1}p}{\psi - \sumlim{k=1}{N} c_k \phi_k} \cdot \cE^j \Phi(0)  \\
& = \psh{A^{-1}p}{\psi - \sumlim{k=1}{N} c_k \phi_k} \cdot \cE^j \Phi(0)  \\
& = \psh{A^{-1}p}{\psi_e - \sumlim{k=1}{N} c_k \phi_k} \cdot \cE^j \Phi(0) \ ,
\end{align*}
as $p$ is even.

Combining Lemmas \ref{lem:ApPhi} and \ref{lem:stab}, we get
\begin{equation*}
(\cE^j \Phi(0))^T A^{-1} p = \rho(t/ \eta) (\cE^j \Phi(0))^T (C_\eta G_\eta)^{-1} C_\eta Q(t/\eta) = \rho(t / \eta) x_\eta^T Q(t/ \eta) \ ,
\end{equation*}
where $\| x_\eta \| = \cO(\eta^{-2j})$. 

Using again Lemma \ref{lem:approx}, we obtain 
\begin{align*}
\left| \psh{A^{-1}p}{\psi - \sumlim{k=1}{N} c_k \phi_k} \cdot \cE^j \Phi(0) \right| & \leq C \| x_\eta\| \left\| \rho(t) \sumlim{\ell=0}{N-1} |Q_\ell(t)| \right\|_{L^1[-1,1]} \left\| \psi - \sumlim{k=1}{N} c_k \phi_k \right\|_{L^\infty[-\eta,\eta]} \\
& \leq  C \eta^{2N-2j}.
\end{align*}

We therefore obtain
\begin{equation*}
| [\tilde{\psi}^{(2j+1)}]_0 | \leq C \eta^{2N-2j}.
\end{equation*}

For $j \geq N$, we then have 
\begin{equation*}
(-E)^j \psi(0) - \sumlim{k=1}{N} c_k (-\epsilon_k)^j \phi_k(0) \not= 0, 
\end{equation*}
and by Lemmas \ref{lem:approx} and \ref{lem:stab}, 
\begin{equation*}
\left| \psh{A^{-1}p}{\psi - \sumlim{j=1}{N} c_j \phi_j} \cdot \cE^j \Phi(0) \right| \leq C \underbrace{\| (C_\eta G_\eta)^{-1} C_\eta \|}_{=\cO(\eta^{2-2N})} \underbrace{\left\| \psi - \sumlim{k=1}{N} c_k \phi_k \right\|_{L^\infty[-\eta,\eta]}}_{=\cO(\eta^{2N})} \leq C \eta^2.
\end{equation*}
We therefore have
\begin{equation*}
[\tilde{\psi}^{(2j+1)}]_0 = \underbrace{(-E)^j \psi(0) - \sumlim{k=1}{N} c_k (-\epsilon_k)^j \phi_k(0)}_{\not= 0} + \underbrace{\psh{A^{-1}p}{\psi - \sumlim{k=1}{N} c_k \phi_k} \cdot \cE^j \Phi(0)}_{= \cO(\eta^2)} . \\
\end{equation*}
Thus,
\begin{equation*}
|[\tilde{\psi}^{(2j+1)}]_0| \leq C,
\end{equation*}
which completes the proof.
\end{proof}

\subsection{$d$-th derivative jump}
 
We use the notation introduced in the previous section.

\begin{proof}[Proof of Proposition \ref{lem:der_jump_eta}]
We give the proof only for $k=d$ as the proof for $d+1 \leq k \leq 2d-2$ is very similar. By definition of $\tilde{\psi}$ and $\widetilde{\Phi}$,
\begin{align*}
[\tilde{\psi}^{(d)}]_\eta & = \psh{\tilde{p}}{\tilde{\psi}} \cdot [\widetilde{\Phi}^{(d)} ]_\eta  \\
& = \frac{1}{\eta^d} \psh{\tilde{p}}{\tilde{\psi}} \cdot ( C_\eta^{(P)} P^{(d)}(1) - \eta^d \Phi^{(d)}(\eta))  \\
& = \frac{1}{\eta^d} \psh{A^{-1}p}{\psi} \cdot ( C_\eta^{(P)} P^{(d)}(1) - \eta^d \Phi^{(d)}(\eta))  \qquad \text{(by Equation (\ref{eq:tilde_p-Ap-psi}))}, \\
& = \frac{1}{\eta^d} \itg{-1}{1}{\rho(t) \psi(\eta t) Q(t)}{t} \cdot C_\eta^T (G_\eta^T C_\eta^T)^{-1} ( C_\eta^{(P)} P^{(d)}(1) - \eta^d \Phi^{(d)}(\eta)) \ \text{(by Lemma \ref{lem:ApPhi})}.
\end{align*}

We know from \eqref{eq:columns_Ceta} that the columns of $C_\eta^{(P)}$ are linear combinations of $\eta^k \Phi^{(k)}(\eta)$. Let us apply Lemma \ref{lem:decomp} to $\Phi$. As the remainder is $C^{2d-2}$, for $k \leq d-1$, we can differentiate $k$ times and we have for $k$ even :
\begin{equation*}
\eta^k \Phi^{(k)}(\eta) = \sumlim{j= \frac{k}{2} }{N-1} \left( \frac{\eta^{2j}}{(2j-k)!} - \frac{Z_0}{2} \frac{\eta^{2j+1}}{(2j+1-k)!}  \right)  \cE^j \Phi(0) + \cO(\eta^{2N}),
\end{equation*}
and for $k$ odd, we have
\begin{equation*}
\eta^k \Phi^{(k)}(\eta) = -\frac{Z_0}{2} \eta^{k} D^{\frac{k-1}{2}}\Phi(0) + \sumlim{j= \frac{k+1}{2} }{N-1} \left( \frac{\eta^{2j}}{(2j-k)!} - \frac{Z_0}{2} \frac{\eta^{2j+1}}{(2j+1-k)!} \right) \cE^j \Phi(0) + \cO(\eta^{2N}).
\end{equation*}
But by Lemma \ref{lem:stab}, for $0 \leq k \leq 2N-2$, $k$ even, we have :
\begin{align}
C_\eta^T (G_\eta^T C_\eta^T)^{-1} \eta^k \Phi^{(k)}(\eta) & = \sumlim{j= \frac{k}{2} }{N-1} \left( \frac{\eta^{2j}}{(2j-k)!} - \frac{Z_0}{2} \frac{\eta^{2j+1}}{(2j+1-k)!}  \right) C_\eta^T (G_\eta^T C_\eta^T)^{-1} \cE^j \Phi(0) + \cO(\eta^{2}) \nonumber \\
& = \sumlim{j= \frac{k}{2} }{N-1} \frac{(2j)!}{(2j-k)!} \begin{pmatrix} I_N \\ M_\pi^T \end{pmatrix}  e_j + \cO(\eta).
\label{eqlem:GetaPhi_even}
\end{align}
Similarly for $0 \leq k \leq 2N-1$, $k$ odd :
\begin{align}
C_\eta^T (G_\eta^T C_\eta^T)^{-1} \eta^k \Phi^{(k)}(\eta) & = - \frac{Z_0}{2} C_\eta^T (G_\eta^T C_\eta^T)^{-1} \eta^{k} D^{\frac{k-1}{2}}\Phi(0) \nonumber \\
&  + \sumlim{j= \frac{k+1}{2} }{N-1} \left( \frac{\eta^{2j}}{(2j-k)!} - \frac{Z_0}{2} \frac{\eta^{2j+1}}{(2j+1-k)!} \right) C_\eta^T (G_\eta^T C_\eta^T)^{-1} \cE^j \Phi(0) + \cO(\eta^{2}) \nonumber \\
& =  \sumlim{j= \frac{k+1}{2} }{N-1} \frac{(2j)!}{(2j-k)!} \begin{pmatrix} I_N \\ M_\pi^T \end{pmatrix} e_j + \cO(\eta) \ ,
\label{eqlem:GetaPhi_odd}
\end{align}
and for $k \geq 2N$, using $\| C_\eta^T (G_\eta^T C_\eta^T)^{-1} \| = \cO(\eta^{2-2N}) $ we have
\begin{equation*}
C_\eta^T (G_\eta^T C_\eta^T)^{-1} \eta^k \Phi^{(k)}(\eta) = \cO(\eta).
\end{equation*}
We have proved that $[\tilde{\psi}^{(d)}]_\eta = \cO(\eta^{-d})$ but it is possible to have a slightly better estimate. 

Observing that $\psi$ is in fact a Lipschitz function and not only a continuous function, we have for $|t| \leq 1$ :
\begin{equation*}
\psi(\eta t) = \psi(0) + \cO(\eta).
\end{equation*}
By definition of the polynomials $Q_k$, we have
\begin{align*}
\itg{-1}{1}{\rho(t) Q(t) \psi(\eta t)}{t} & = \psi(0) \itg{-1}{1}{\rho(t) Q(t)}{t} + \cO(\eta)  \\
& = \psi(0) e_0 + \cO(\eta).
\end{align*}
To complete the proof of the proposition, it remains to show 
\begin{equation*}
e_0 \cdot C_\eta^T (G_\eta^T C_\eta^T)^{-1} ( C_\eta^{(P)} P^{(d)}(1) - \eta^d \Phi^{(d)}(\eta)) = \cO(\eta).
\end{equation*}
As we have for $j \geq 1$ 
\begin{equation*}
e_0^T \begin{pmatrix} I_N \\ M_\pi^T \end{pmatrix} e_j = 0,
\end{equation*}
then for $d \geq 2$, equations (\ref{eqlem:GetaPhi_even}) and (\ref{eqlem:GetaPhi_odd}) lead to
\begin{equation*}
e_0 \cdot C_\eta^T (G_\eta^T C_\eta^T)^{-1} \eta^d \Phi^{(d)}(\eta) = \cO(\eta).
\end{equation*}
Recall that the columns of $C_\eta^{(P)}$ satisfy the relation
\begin{equation*}
\forall 0 \leq j \leq d-1, \  c_j = \eta^j \Phi^{(j)}(\eta) - \sumlim{k=0}{j-1} P_k^{(j)}(1) c_k.
\end{equation*}
But $P^{(j)}_0(1)=0$ for  $j \geq 1$ so in fact, for all $k\geq 1$, $c_k$ is a linear combination of the vectors $\eta^j \Phi^{(j)}(\eta)$ for $ 1 \leq j \leq k$. Moreover by definition, we have 
\begin{equation*}
P^{(d)}(1) = (\underbrace{0, \dots, 0}_{\lfloor \frac{d}{2} \rfloor}, *, \dots, *)^T ,
\end{equation*}
so $C_\eta^{(P)} P^{(d)}(1)$ is a linear combination of the last $\lceil \frac{d}{2} \rceil$ columns of $C_\eta^{(P)}$. Thus $C_\eta^{(P)} P^{(d)}(1)$ is a linear combination of the vectors $\eta^j \Phi^{(j)}(\eta)$ for $ 1 \leq j \leq d-1$ and therefore in view of (\ref{eqlem:GetaPhi_even}) and (\ref{eqlem:GetaPhi_odd}), we have
\begin{equation*}
e_0^T C_\eta^T (G_\eta^T C_\eta^T)^{-1} C_\eta^{(P)} P^{(d)}(1) = \cO(\eta).
\end{equation*}

\end{proof}

\begin{proof}[Proof of Theorem~\ref{theo:FourierPsi}]
First, we need to bound the remainder $\itg{0}{1}{\tilde{\psi}^{(2d-1)}(x)e^{-2i\pi mx}}{x}$ with respect to $\eta$. $\widetilde{\Phi}$ is a vector of polynomials of degree at most $2d-2$, thus $\widetilde{\Phi}^{(2d-1)} = 0$. Thus
\begin{equation*}
\begin{split}
\left| \itg{0}{1}{\tilde{\psi}^{(2d-1)}(x)e^{-2i\pi mx}}{x} \right| & \leq \itg{0}{1}{|\psi^{(2d-1)}(x)|}{x} + \itg{-\eta}{\eta}{|\psh{\tilde{p}}{\tilde{\psi}} \cdot \Phi^{(2d-1)}(x)|}{x} \\
& + \itg{a-\eta}{a+\eta}{|\psh{\tilde{p}^a}{\tilde{\psi}} \cdot \Phi^{(2d-1)}(x-a)|}{x}.
\end{split}
\end{equation*}
We have $\psh{\tilde{p}}{\tilde{\psi}} = \psh{A^{-1}p}{\psi}$ by \eqref{eq:tilde_p-Ap-psi} and by Lemmas \ref{lem:ApPhi} and \ref{lem:stab}, 
$$
|\psh{A^{-1}p}{\psi}| \leq \frac{C}{\eta^{2N-2}},
$$
where $C$ is a positive constant independent of $\eta$. Thus
\begin{align*}
\itg{-\eta}{\eta}{|\psh{\tilde{p}}{\tilde{\psi}} \cdot \Phi^{(2d-1)}(x)| }{x} & \leq \frac{C}{\eta^{2N-3}}.
\end{align*}
Then by Proposition \ref{prop:Fourier_decay}, using the estimates \eqref{eq:der_jump_1} and \eqref{eq:der_jump_2} on the derivative jumps 
\begin{align*}
|\widehat{\tilde{\psi}}_m| & \leq C \left( \sumlim{j=0}{\lfloor \frac{d}{2} \rfloor -1} \frac{1}{(2\pi m)^{2+2j}} \left| \left[\tilde{\psi}^{(2j+1)}\right]_0 \right| + \sumlim{k=d}{2d-2} \frac{1}{(2\pi m)^{k+1}} \left| \left[\tilde{\psi}^{(k)} \right]_{\pm\eta} \right|  \right. \\
&  + \sumlim{j=0}{\lfloor \frac{d}{2} \rfloor -1} \frac{1}{(2\pi m)^{2+2j}} \left| \left[\tilde{\psi}^{(2j+1)}\right]_0 \right| + \sumlim{k=d}{2d-2} \frac{1}{(2\pi m)^{k+1}} \left| \left[\tilde{\psi}^{(k)} \right]_{a\pm\eta} \right| \\
&  + \frac{1}{(2\pi m)^{2d-1}} \left|\itg{0}{1}{\tilde{\psi}^{(2d-1)}(x)e^{-2i\pi mx}}{x} \right| \Bigg) \\
& \leq C \left( \sumlim{j=0}{\lfloor \frac{d}{2} \rfloor -1} \frac{\eta^{2N-2j}}{m^{2+2j}} + \sumlim{k=d}{2d-2} \frac{1}{\eta^{k-1} m^{k+1}} + \frac{1}{\eta^{2N-3} m^{2d-1}}\right) .
\end{align*}
Since $N \leq d$ and $m \geq \tfrac{1}{\eta}$, we have the result. 
\end{proof}

\subsection{Error bound on the eigenvalues}

To derive the estimate on the eigenvalues, we would like to use the following classical result (\cite{weinberger1974variational}, p. 68).

\begin{prop}
\label{prop:eig_error}
Let $H$ be a self-adjoint coercive $H^1$-bounded operator, $E_1 \leq \dots \leq E_n$ be the lowest eigenvalues of $H$ and $\psi_1, \dots, \psi_n$ be $L^2$-normalized associated eigenfunctions. \linebreak Let $E^{(M)}_1 \leq \dots \leq E^{(M)}_n$ be the lowest eigenvalues of the Rayleigh quotient of $H$ restricted to the subspace $V_M$ of dimension $M$. 

Let $w_k \in V_M$ for $1 \leq k \leq n$ be such that 
\begin{equation*}
\sumlim{k=1}{n} \| w_k - \psi_k \|^2_{H^1} < 1.
\end{equation*}

Then there exists a positive constant $C$ which depends on the $H^1$ norm of $H$ and the coercivity constant such that for all $1 \leq k \leq n$
\begin{equation*}
\left| E^{(M)}_k - E_k \right| \leq C \sumlim{k=1}{n} \| w_k - \psi_k \|^2_{H^1} .
\end{equation*}
\end{prop}

We would like to apply this result to $\psi_M = (\mathrm{Id}+T) \tilde{\psi}_M$ where $\tilde{\psi}_M$ is the truncation of $\tilde{\psi}$ to the first $M$ plane-waves but we need to bound the $H^1$ norm of $T$ with respect to $\eta$. Coercivity for our one-dimensional model has been proved in \cite{cances2017discretization}. To find this bound, we will need to rewrite $Tf$ in a convenient way.

\begin{lem}
\label{lem:Tconvenient}
For $f \in H^1_\mathrm{per}(0,1)$, we have for $|x| \leq \eta$:
\begin{multline*}
Tf(x) = \psh{\tilde{p}}{f} \cdot \left( \Phi(x) - \widetilde{\Phi}(x)\right) \\
 = \left(C_\eta^{(P)}\right)^T \left( C_\eta^{(P)} G(P) \left(C_\eta^{(P)}\right)^T \right)^{-1} C_\eta^{(P)}\itg{-1}{1}{\rho(t) f(\eta t) P(t)}{t} \cdot \left( \begin{pmatrix} C_1^{-1} \\ 0 \end{pmatrix} \Phi(x) - P(x/\eta) \right) ,
\end{multline*}
where $G(P)$ is the following Gram matrix :
\begin{equation*}
G(P) = \itg{-1}{1}{\rho(t) P(t) P(t)^T}{t} ,
\end{equation*}
and $C_1 \in \R^{N \times N}$ is the square matrix defined in Lemma \ref{lem:C_eta}.
\end{lem}

\begin{proof}
For $|x| \leq \eta$, we have :
\begin{align*}
(Tf)(x) & = \psh{\tilde{p}}{f} \cdot (\Phi(x) - \widetilde{\Phi}(x)) \\
& = \psh{B^{-1}p}{f} \cdot (\Phi(x) - \widetilde{\Phi}(x)) .
\end{align*}
Recall that
\begin{align*}
B & = \psh{p}{\widetilde{\Phi}^T}  \\
& = \psh{\rho(t/\eta) \widetilde{\Phi}}{\widetilde{\Phi}^T}  \\
& = \eta \itg{-1}{1}{\rho(t) C_\eta^{(P)} P(t) P(t)^T \left(C_\eta^{(P)}\right)^T}{t}  \\
& = \eta C_\eta^{(P)} G(P) \left(C_\eta^{(P)}\right)^T.
\end{align*}
Thus,
\begin{align*}
\psh{B^{-1}p}{f} \cdot  \widetilde{\Phi}(x) & = \left( C_\eta^{(P)} G(P) \left(C_\eta^{(P)}\right)^T \right)^{-1} C_\eta^{(P)} \itg{-1}{1}{\rho(t) f(\eta t) P(t)}{t} \cdot C_\eta^{(P)} P(x/\eta) \\
& = \left(C_\eta^{(P)}\right)^T \left( C_\eta^{(P)} G(P) \left(C_\eta^{(P)}\right)^T \right)^{-1} C_\eta^{(P)} \itg{-1}{1}{\rho(t) f(\eta t) P(t)}{t} \cdot P(x/\eta).
\end{align*}
Using the identity 
\begin{equation*}
C_\eta^{(P)} \left( \begin{array}{c}
C_1^{-1} \\ \hline  0 
\end{array} \right) = \Big( C_1 \ | \ C_2 \Big) \left( \begin{array}{c}
C_1^{-1} \\ \hline  0 
\end{array} \right) = I_N ,
\end{equation*}
we can formally rewrite $\psh{B^{-1}p}{f}$ as
\begin{equation}
\label{eq:T_Phi}
\psh{B^{-1}p}{f} \cdot  \Phi(x)  = \left(C_\eta^{(P)}\right)^T \left(C_\eta^{(P)} G(P) \left(C_\eta^{(P)}\right)^T\right)^{-1} C_\eta^{(P)} \itg{-1}{1}{\rho(t) f(\eta t) P(t)}{t} \cdot \left( \begin{array}{c}
C_1^{-1} \\ \hline  0 
\end{array} \right) \Phi(x) ,
\end{equation}
and the result follows.
\end{proof}

\begin{lem}
\label{lem:Tf_L2}
There exists a positive constant $C$ independent of $\eta$ such that for all \linebreak $f \in H^1_\mathrm{per}(0,1)$ and for all $x \in \R$, we have
\begin{equation*}
\forall \, 0 < \eta \leq \eta_0, \ \left| \psh{\tilde{p}}{f} \cdot \left( \Phi(x) - \widetilde{\Phi}(x)\right) \right| \leq C \eta \| f \|_{H^1_\mathrm{per}} .
\end{equation*}
\end{lem}

\begin{proof} In this proof, $C$ denotes a generic constant that does not depend on $\eta$ or $f$. Let $f\in H^1_\mathrm{per}(0,1)$. On $|x| \geq \eta$, we have by definition of $\widetilde{\Phi}$ 
\begin{equation*}
\psh{\tilde{p}}{f} \cdot \left( \Phi(x) - \widetilde{\Phi}(x)\right) = 0.
\end{equation*}
We deduce from Lemma \ref{lem:Tconvenient} that for $|x| \leq \eta$ we have 
\begin{multline*}
\psh{\tilde{p}}{f} \cdot \left( \Phi(x) - \widetilde{\Phi}(x)\right) = \left(C_\eta^{(P)}\right)^T \left( C_\eta^{(P)} G(P) \left(C_\eta^{(P)}\right)^T \right)^{-1} C_\eta^{(P)}\itg{-1}{1}{\rho(t) f(\eta t) P(t)}{t} \\
\cdot \left( \left( \begin{array}{c}
C_1^{-1} \\ \hline  0 
\end{array} \right) \Phi(x) - P(x/\eta) \right).
\end{multline*}
The proof of the lemma consists of four steps. We will successively show that 
\begin{enumerate}
\item $\left( \begin{array}{c}
C_1^{-1} \\ \hline  0 
\end{array} \right) \Phi(x) - P(x/\eta) = \begin{pmatrix} 0 \\ * \end{pmatrix} + \cO(\eta)$, where $\begin{pmatrix} 0 \\ * \end{pmatrix}$ is uniformly bounded in $\eta$ and $x$;
\item $\itg{-1}{1}{\rho(t) f(\eta t) P(t)}{t} = f(0) G(P) e_0 + \cO(\eta) \|f\|_{H^1_\mathrm{per}}$;
\item the norm of $\left(C_\eta^{(P)}\right)^T \left( C_\eta^{(P)} G(P) \left(C_\eta^{(P)}\right)^T \right)^{-1} C_\eta^{(P)}$ is uniformly bounded in $\eta$;
\item for $j~\geq~1$, $e_j^T \left(C_\eta^{(P)}\right)^T \left( C_\eta^{(P)} G(P) \left(C_\eta^{(P)}\right)^T \right)^{-1} C_\eta^{(P)} G(P) e_0$ is of order $\cO(\eta)$.
\end{enumerate}
Indeed, assuming these statements hold, we can infer from statement 2 that
\begin{multline*}
\left| \psh{\tilde{p}}{f} \cdot \left( \Phi(x) - \widetilde{\Phi}(x)\right) \right| \\
\leq |f(0)| \left| \left(C_\eta^{(P)}\right)^T \left( C_\eta^{(P)} G(P) \left(C_\eta^{(P)}\right)^T \right)^{-1} C_\eta^{(P)} G(P) e_0 \cdot \left( \left( \begin{array}{c}
C_1^{-1} \\ \hline  0 
\end{array} \right) \Phi(x) - P(x/\eta) \right) \right| \\
+ \cO(\eta) \|f\|_{H^1_\mathrm{per}} \left| \left(C_\eta^{(P)}\right)^T \left( C_\eta^{(P)} G(P) \left(C_\eta^{(P)}\right)^T \right)^{-1} C_\eta^{(P)} \left( \left( \begin{array}{c}
C_1^{-1} \\ \hline  0 
\end{array} \right) \Phi(x) - P(x/\eta) \right) \right|.
\end{multline*}
We treat both terms separately. For the second term, by statements 1 and 3, we have
\begin{equation*}
\left| \left(C_\eta^{(P)}\right)^T \left( C_\eta^{(P)} G(P) \left(C_\eta^{(P)}\right)^T \right)^{-1} C_\eta^{(P)} \left( \left( \begin{array}{c}
C_1^{-1} \\ \hline  0 
\end{array} \right) \Phi(x) - P(x/\eta) \right) \right| \leq C.
\end{equation*}
For the first one, by statement 1, we only have to check that for $j \geq 1$, we have
\begin{equation*}
\left| \left(C_\eta^{(P)}\right)^T \left( C_\eta^{(P)} G(P) \left(C_\eta^{(P)}\right)^T \right)^{-1} C_\eta^{(P)} G(P) e_0 \cdot e_j \right| \leq C \eta,
\end{equation*}
which is exactly statement 4. The lemma is then proved using the Sobolev embedding \linebreak $\| f \|_{L^\infty_\mathrm{per}} \leq C \| f \|_{H^1_\mathrm{per}}$. 

\paragraph{Step 1}
Writing down the Taylor expansion of $\Phi$ at $\eta$, we obtain
\begin{align*}
\Phi(x) & = \sumlim{k=0}{N-1} \frac{(x-\eta)^k}{k!} \Phi^{(k)}(\eta) + \cO((x-\eta)^N)  \\
& = \sumlim{k=0}{N-1} \frac{1}{k!} \left(\frac{x}{\eta} - 1 \right)^k \eta^k \Phi^{(k)}(\eta) + \cO((x-\eta)^N) .
\end{align*}
By Lemma \ref{lem:C_eta}, we have 
\begin{equation*}
C_1^{-1} \eta^k \Phi^{(k)}(\eta) = \mathcal{P} K_1 \eta^k \Phi^{(k)}(\eta) = \mathcal{P} e_k ,
\end{equation*}
and for $k \not= 0$, $\mathcal{P}e_k \cdot e_0 =0$. We also know that $ \| C_1^{-1} \| = \cO(\eta^{1-N})$, so that
\begin{equation*}
C_1^{-1} \Phi(x) = \begin{pmatrix} 1 \\ * \end{pmatrix} + \cO(\eta).
\end{equation*}
By definition $P_0 = 1$, and therefore
\begin{equation*}
\left( \begin{array}{c}
C_1^{-1} \\ \hline  0 
\end{array} \right) \Phi(x) - P(x/\eta) = \begin{pmatrix} 0 \\ * \end{pmatrix} + \cO(\eta).
\end{equation*}

\paragraph{Step 2}
Since $f \in H^1_\mathrm{per}(0,1)$, by the Sobolev embedding theorem, $f$ is continuous and $f(0)$ exists. Thus we can write
\begin{align*}
\itg{-1}{1}{\rho (t) f(\eta t) P(t)}{t} & = f(0) \itg{-1}{1}{\rho (t) P(t)}{t} + \itg{-1}{1}{\rho (t) (f(\eta t)-f(0)) P(t)}{t}  \\
& = f(0) G(P) e_0 + \itg{-1}{1}{\rho (t) (f(\eta t)-f(0)) P(t)}{t} ,
\end{align*}
since $P_0 = 1$. Using 
$$
f(\eta t) = f(0) + \itg{0}{\eta t}{f'(x)}{x}, 
$$ 
and Cauchy-Schwarz inequality, we obtain
\begin{align*}
\left| \itg{-1}{1}{\rho (t) (f(\eta t)-f(0)) P(t)}{t} \right| & \leq \left( \itg{-1}{1}{\rho(t)^2 P^2(t)}{t} \itg{-1}{1}{\left(\itg{0}{\eta t}{f'(x)}{x}\right)^2}{t} \right)^{1/2} \\
& \leq C \left( \itg{-1}{1}{\left(\itg{0}{\eta t}{f'(x)^2}{x}\right) \eta^2 t^2 }{t} \right)^{1/2} \\
& \leq C \eta \|f\|_{H^1_\mathrm{per}} .
\end{align*}

\paragraph{Step 3}
We want to bound the norm of the matrix 
\begin{equation*}
\left(C_\eta^{(P)}\right)^T \left( C_\eta^{(P)} G(P) \left(C_\eta^{(P)}\right)^T \right)^{-1} C_\eta^{(P)}.
\end{equation*}
Since $G(P)$ is the Gram matrix of the polynomials $P_k$ for the weight $\rho$, $G(P)$ is a symmetric positive definite matrix and thus admits a square root. It is easy to check that 
\begin{equation*}
G(P)^{1/2} \left(C_\eta^{(P)}\right)^T \left( C_\eta^{(P)} G(P) \left(C_\eta^{(P)}\right)^T \right)^{-1} C_\eta^{(P)} G(P)^{1/2} 
\end{equation*}
is an orthogonal projector. Its norm is therefore uniformly bounded in $\eta$. 

\paragraph{Step 4}
Let $G_1 \in \R^{N \times N}$, $G_2 \in \R^{N \times (d-N)}$ and $G_3 \in \R^{(d-N) \times (d-N)}$ be the matrices respectively defined by
\begin{equation*}
G(P) = \begin{pmatrix}
G_1 & G_2 \\
G_2^T & G_3
\end{pmatrix}.
\end{equation*}
Let $M_\eta$ be the matrix 
\begin{equation*}
C_1^{-1} C_2 = M_\eta.
\end{equation*}
Recall that by Lemma \ref{lem:C_eta}, $\| M_\eta \| \leq C$. With this notation, we have
\begin{equation*}
C_\eta^{(P)} G(P) \left(C_\eta^{(P)}\right)^T = C_1 \left( G_1 + M_\eta G_2^T + G_2 M_\eta^T + M_\eta G_3 M_\eta^T \right) ,
\end{equation*}
and therefore, 
\begin{multline}
\label{eq:moche}
\left(C_\eta^{(P)}\right)^T \left( C_\eta^{(P)} G(P) \left(C_\eta^{(P)}\right)^T \right)^{-1} C_\eta^{(P)} e_0 \\
= \begin{pmatrix} I_N \\ M_\eta \end{pmatrix} \left( G_1 + M_\eta G_2^T + G_2 M_\eta^T + M_\eta G_3 M_\eta^T \right)^{-1} \left( G_1 + M_\eta G_2^T \right) e_0 .
\end{multline}
We will now show that $M_\eta^T e_0 = \cO(\eta)$. By definition,
\begin{equation*}
e_0^T M_\eta = e_0^T C_1^{-1} C_2 = e_0^T \mathcal{P} K_1 C_2 .
\end{equation*}
By Lemma \ref{lem:C_eta}, $e_0^T \mathcal{P} = e_0$ and by definition of $K_1$, $e_0^T K_1 = g_0^T$ where $g_0$ is the vector satisfying $g_0^T \eta^k \Phi^{(k)} (\eta) = \delta_{0k}$ for $k \leq N-1$.
Again by Lemma \ref{lem:C_eta}, the columns of $C_\eta^{(P)}$ satisfy : 
\begin{equation*}
\forall 0 \leq j \leq d-1, \  c_j = \eta^j \Phi^{(j)}(\eta) - \sumlim{k=0}{j-1} P_k^{(j)}(1) c_k ,
\end{equation*}
with $P_0^{(j)}(1) =0$ for $j \geq 1$. Consequently $c_j$ is a linear combination of the vectors $\eta^k \Phi^{(k)}(\eta)$ for $1 \leq k \leq j$. Since $\|g_0\| = \cO(\eta^{1-N})$, we get $g_0^T C_2 = \cO(\eta)$.
Coming back to Equation (\ref{eq:moche}), we have 
\begin{equation*}
\left(C_\eta^{(P)}\right)^T \left( C_\eta^{(P)} G(P) \left(C_\eta^{(P)}\right)^T \right)^{-1} C_\eta^{(P)} e_0 = \begin{pmatrix} I_N \\ M_\eta^T \end{pmatrix} e_0 + \cO(\eta).
\end{equation*}
Consequently,
\begin{align*}
e_j^T \left(C_\eta^{(P)}\right)^T \left( C_\eta^{(P)} G(P) \left(C_\eta^{(P)}\right)^T \right)^{-1} C_\eta^{(P)} e_0 & = e_j^T \begin{pmatrix} I_N \\ M_\eta^T \end{pmatrix} e_0 + \cO(\eta) \\
& = \begin{cases}
\cO(\eta), & \ 0 \leq j \leq N-1 , \\
e_{j-N}^T M_\eta^T e_0 = \cO(\eta), & \ N \leq j \leq d-1.
\end{cases}
\end{align*}
\end{proof}

%\begin{note}
%In fact, we have proved that for $|x| \leq \eta$,
%\begin{equation}
%\left| \psh{\tilde{p}}{f} \cdot \widetilde{\Phi}(x) \right| \leq C \|f\|_{L^\infty(-\eta, \eta)},
%\end{equation}
%and
%\begin{equation}
%\left| \psh{\tilde{p}}{f} \cdot \Phi(x) \right| \leq C \|f\|_{L^\infty(-\eta, \eta)}.
%\end{equation}
%These estimates will be useful when we will bound the eigenvalues of the PAW equations in Section \ref{sec:truncated_PAW}. 
%\end{note}

We can establish a similar result for $\psh{\tilde{p}}{f} \cdot \left( \Phi'(x) - \widetilde{\Phi}'(x)\right)$.
\begin{lem}
\label{lem:Tf_prime}
There exists a positive constant $C$ independent of $\eta$ such that for all \linebreak $f \in H^1_\mathrm{per}(0,1)$ and for all $x \in \R$, we have\begin{equation*}
\forall \, 0 < \eta \leq \eta_0, \  \left| \psh{\tilde{p}}{f} \cdot \left( \Phi'(x) - \widetilde{\Phi}'(x)\right) \right| \leq C \| f \|_{H^1_\mathrm{per}} , 
\end{equation*}
\end{lem}

\begin{proof}
It is a transposition of the proof of the previous lemma. The first step is simply replaced by 
\begin{enumerate}
\item $\left( \begin{array}{c}
C_1^{-1} \\ \hline  0 
\end{array} \right) \Phi'(x) - \frac{1}{\eta} P'(x/\eta) = \frac{1}{\eta} \begin{pmatrix} 0 \\ * \end{pmatrix} + \cO(1)$, where $\begin{pmatrix} 0 \\ * \end{pmatrix}$ is uniformly bounded in $\eta$  and $x$.
\end{enumerate}
To prove the latter statement, we observe that $P_0' = 0$ and by a Taylor expansion of $\Phi'$, we obtain 
\begin{equation*}
\Phi'(x) = \frac{1}{\eta} \sumlim{k=1}{N-1} \frac{1}{(k-1)!} \left(\frac{x}{\eta} - 1 \right)^{k-1} \eta^k \Phi^{(k)}(\eta) + \cO(\eta^{N-1}) ,
\end{equation*}
hence 
\begin{equation*}
C_1^{-1} \Phi'(x) = \frac{1}{\eta} \begin{pmatrix} 0 \\ * \end{pmatrix} + \cO(1).
\end{equation*}
\end{proof}

\begin{lem}
\label{lem:normeT}
There exists a positive constant $C$ independent of $\eta$ such that for all $f \in H^1_\mathrm{per}(0,1)$, we have
\begin{equation*}
\forall \, 0 < \eta \leq \eta_0, \ \|Tf\|_{H^1_\mathrm{per}} \leq C \eta^{\frac{1}{2}} \|f\|_{H^1_\mathrm{per}}
\end{equation*}
\end{lem}

\begin{proof}
This is a straightforward consequence of Lemmas \ref{lem:Tf_L2} and \ref{lem:Tf_prime}. 
\end{proof}

\begin{proof}[Proof of Theorem \ref{theo:energies}]
Applying Proposition \ref{prop:eig_error} to $\psi_M = (\mathrm{Id}+T) \Pi_M \tilde{\psi}$, where $\Pi_M$ is the truncation to the first $M$ plane-waves, we have
\begin{align*}
|E_M^\eta - E| & \leq C \|(\mathrm{Id}+T)(\Pi_M \tilde{\psi} - \tilde{\psi}) \|_{H^1_\mathrm{per}}^2 \\
& \leq C ( \| \Pi_M \tilde{\psi} - \tilde{\psi} \|_{H^1_\mathrm{per}}^2 + \|T(\Pi_M \tilde{\psi} - \tilde{\psi} ) \|_{H^1_\mathrm{per}}^2 ).
\end{align*}
By Lemma \ref{lem:normeT}, 
\begin{align*}
\|T(\Pi_M \tilde{\psi} - \tilde{\psi}) \|_{H^1_\mathrm{per}}^2 \leq C \| \Pi_M \tilde{\psi} - \tilde{\psi} \|_{H^1_\mathrm{per}}^2 \ ,
\end{align*}
and we deduce from Theorem \ref{theo:FourierPsi} that
\begin{align*}
|E_M^\eta - E| &  \leq C \| \Pi_M \tilde{\psi} - \tilde{\psi} \|_{H^1_\mathrm{per}}^2  \\
& \leq C \sumlim{j=M+1}{\infty} (1+j^2) | \widehat{\tilde{\psi}}_j|^2  \\
& \leq C \sumlim{j=M+1}{\infty} \left( \frac{\eta^{4N}}{j^2} + \frac{1}{\eta^{2d-2}j^{2d}} \right)  \\
& \leq C \left( \frac{\eta^{4N}}{M} +  \frac{1}{\eta^{2d-2}M^{2d-1}} \right).
\end{align*}
\end{proof}

\section{Perturbation by a continuous potential}
\label{sec:smooth}

In standard electronic structure calculations, the Hartree and exchange-correlation terms are modelled by a potential that is smoother than the Coulomb potential. To reproduce this setting in our one-dimensional toy model, a smoother potential $W$ is added to the Hamiltonian \eqref{eq:H_mol}. In the following, we examine how the VPAW method accelerates the computation of eigenvalues. 

Consider the Hamiltonian 
\begin{equation}
H = -\frac{\mathrm{d}^2}{\mathrm{d}x^2} - Z_0 \sumlim{k \in \Z}{} \delta_k - Z_a \sumlim{k \in \Z}{} \delta_{a+k} + W \ ,
\label{eq:new_H}
\end{equation}
where $W$ is $1$-periodic, continuous, $0 < a < 1$, $Z_0, Z_a >0$. 

With the VPAW method, the generalized eigenvalue problem
\begin{equation}
\label{eq:newH_VPAW}
(\mathrm{Id} + T^*) H (\mathrm{Id}+T) \tilde{\psi} = E (\mathrm{Id} + T^*) (\mathrm{Id}+T) \tilde{\psi},
\end{equation}
is solved by expanding $\tilde{\psi}$ in plane-waves. Like in Section \ref{subsec:VPAW1D}, $T = T_0 + T_a$, where $T_0$ and $T_a$ act on two disjoint regions $\bigcup\limits_{k \in \Z}[-\eta+k,\eta+k]$ and $\bigcup\limits_{k \in \Z}[a-\eta+k,a+\eta+k]$ respectively. The atomic wave functions $(\phi_k)_{0 \leq k \leq N-1}$ are the non-smooth solutions of the atomic Hamiltonian 
\begin{equation*}
H_0 = -\frac{\mathrm{d}^2}{\mathrm{d}x^2} - Z_0 \sumlim{k \in \Z}{} \delta_k + V \ ,
\end{equation*}
where $V$ can be different from $W$. The eigenvalues associated to $(\phi_k)_{0 \leq k \leq N-1}$ are denoted by $\epsilon_k$. To define the pseudo wave functions $(\tilde{\phi}_k)_{0 \leq k \leq N-1}$ and the projectors $(\tilde{p}_k)_{0 \leq k \leq N-1}$, we proceed as in Section \ref{subsec:VPAW1D}.
\newline 

It follows from the study of the double Dirac delta potential Hamiltonian that the key lemma of the analysis is the structure Lemma \ref{lem:decomp}, which describes the behavior of eigenfunctions near the singularities. It is possible to establish a similar result for the eigenfunctions of the Hamiltonian \eqref{eq:new_H}. 

\begin{lem}
\label{lem:new_decomp}
Let $\psi$ be an eigenfunction of the Hamiltonian $H$ given by \eqref{eq:new_H} for the eigenvalue $E$. Then in a neighborhood of $0$, we have the following expansion :
\begin{multline*}
\psi(x) = \psi(0) \left( \sumlim{j=0}{k} \frac{(-E)^j}{(2j)!} x^{2j} - \frac{Z_0}{2} \frac{(-E)^j}{(2j+1)!} |x|^{2j+1} \right) \\
+ \frac{\psi'(0_+) + \psi'(0_-)}{2} \sumlim{j=0}{k} \frac{(-E)^j}{(2j+1)!} x^{2j+1} + \sumlim{j=0}{k} \underbrace{\idotsint}_{2j+2} W \psi (x) + \psi_{2k+2}(x),
\end{multline*}
where $\psi_{2k+2}$ is a $C^{2k+2}$ function satisfying in a neighbourhood of 0
\begin{equation*}
\begin{cases}
\psi_{2k+2}^{(2k+2)} = (-E)^{k+1} \psi ,\\
\psi_{2k+2}(x) = \cO(x^{2k+2}).
\end{cases}
\end{equation*} 
\end{lem}

\begin{proof}
This lemma can be proved by induction. For $k=0$, we set
\begin{equation*}
\theta_2 (x) = \psi(x) + \frac{Z_0}{2}|x| \psi(0) - \int\limits_0^x \int\limits_0^t W(s) \psi(s) \mathrm{d}s \mathrm{d}t
\end{equation*}
and then proceed as in the proof of Lemma \ref{lem:decomp}.
\end{proof}

We will now make some assumptions on the potentials $V$ and $W$:
\begin{enumerate}
\item $V$ and $W$ are smooth and $1$-periodic;
\item $V$ is even. This property would indeed be satisfied by potentials that does not break the crystal symmetry;
\end{enumerate}

\begin{lem}
\label{lem:struct_V}
For $N \geq 2$:
\begin{equation*}
\iint V \Phi (x)  = \sumlim{k=1}{N-1} (V\Phi)_k \left( \frac{x^{2k}}{(2k)!} - \frac{Z_0}{2} \frac{|x|^{2k+1}}{(2k+1)!} \right) + \cO \left(x^{2N}\right),
\end{equation*}
where $(V \Phi)_k$ is in $\mathrm{span}(\cE^j \Phi(0), \ j \leq N-2)$.
\end{lem}

\begin{proof}
This lemma is proved by induction.
\paragraph{Initialization}
Applying Lemma \ref{lem:new_decomp} with $Z_a = 0$ and $W=V$ to each function $\phi_k$, we obtain expansions of the atomic PAW functions $\phi_k$ in the vicinity of $0$:
$$
\Phi(x) =  \Phi(0) \left(1 - \frac{Z_0}{2} |x|\right) + \cO(x^2).
$$
Deriving twice $\iint V \Phi$ gives
\begin{align*}
\left( \iint V \Phi (x) \right)'' & = V(x) \Phi(x) \\
& = V(0) \Phi(0) \left(1 - \frac{Z_0}{2} |x|\right) + \cO(x^2) .
\end{align*}
Therefore
\begin{equation*}
\iint V \Phi (x) =  V(0) \Phi(0)\left( \frac{x^{2}}{2} - \frac{Z_0}{2} \frac{|x|^{3}}{6} \right) + \cO(x^{4}).
\end{equation*}

\paragraph{Inductive step}
Let us derive twice $\iint V \Phi$:
\begin{align}
\left( \iint V \Phi (x) \right)'' & = V(x) \Phi(x) \nonumber \\
& = \left( \sumlim{k=0}{2N-2} V^{(2k)}(0) \frac{x^{2k}}{(2k)!} \right) \Bigg( \sumlim{k=0}{N-1} \left( \frac{x^{2k}}{(2k)!} - \frac{Z_0}{2} \frac{|x|^{2k+1}}{(2k+1)!} \right) D^k \Phi(0)  \nonumber \\ 
 &  + \sumlim{j=0}{N-2} \underbrace{\idotsint}_{(2j+2)} V \Phi (x) \Bigg) + \cO(x^{2N}).
\label{eq:VPHI}
\end{align}
By the induction hypothesis,
\begin{equation*}
\underbrace{\idotsint}_{(2j+2)} V \Phi (x) = \sumlim{k=j+1}{N-1} (V \Phi)_{k-j} \left( \frac{x^{2k}}{(2k)!}  - \frac{Z_0}{2} \frac{|x|^{2k+1}}{(2k+1)!} \right) + \cO(x^{2N}).
\end{equation*}
Thus,
\begin{align*}
\sumlim{j=0}{N-1} \underbrace{\idotsint}_{(2j+2)} V \Phi (x) & = \sumlim{j=0}{N-2} \sumlim{k=j+1}{N-1} (V \Phi)_{k-j} \left( \frac{x^{2k}}{(2k)!} - \frac{Z_0}{2} \frac{|x|^{2k+1}}{(2k+1)!} \right) \\
& = \sumlim{k=1}{N-1} \sumlim{j=0}{k-1} (V \Phi)_{k-j} \left( \frac{x^{2k}}{(2k)!} - \frac{Z_0}{2} \frac{|x|^{2k+1}}{(2k+1)!} \right).
\end{align*}
Going back to \eqref{eq:VPHI}, expanding the equation and using the last equation, we obtain the result. 
\end{proof}

\begin{lem}
\label{lem:struct_Phi}
In a neighbourhood of $0$, the vector $\Phi$ has the following expansion :
\begin{equation*}
\Phi(x) = \sumlim{j=0}{k} \left( \frac{x^{2j}}{(2j)!}  - \frac{Z_0}{2} \frac{|x|^{2j+1}}{(2j+1)!} \right) X_j + \Phi_{2k+2}(x) \ ,
\end{equation*}
where the function $\Phi_{2k+2}$ is $C^{2k+2}$ at 0 and $X_j$ are vectors satisfying
\begin{equation}
\begin{cases}
X_j \in \mathrm{span}(\cE^\ell \Phi(0) , \ \ell \leq j), \\
X_j - \cE^j \Phi(0) \in \mathrm{span}(\cE^\ell \Phi(0) , \ \ell \leq j-1),
\end{cases}
\label{eq:X_j}
\end{equation}
where $\cE$ is the diagonal matrix with entries $-\epsilon_0, \dots, -\epsilon_{N-1}$. 
\end{lem}

\begin{proof}
We apply Lemmas \ref{lem:new_decomp} and \ref{lem:struct_V} and notice that the vectors $(V \Phi)_k$ are spanned by $(\cE^j \Phi(0) , \ j \leq k-1)$.
\end{proof}

%\begin{note}
%One of the key Lemmas to establish Theorems \ref{theo:FourierPsi} and \ref{theo:energies} is Lemma \ref{lem:stab} which truly depends on the form of the expansion of $\Phi$ around a singularity. It is this very specific form that enables us to give a kind of preconditioner to the matrix $G_\eta = \itg{-1}{1}{\rho(t) Q(t) \Phi(\eta t)^T}{t}$. Without the assumption on the parity of $V$ it would be impossible to establish Lemma \ref{lem:struct_Phi} and thus the strategy adopted to estimate the derivative jumps would not work. It is important to notice that because of the relation (\ref{eq:X_j}), the vectors $X_j$ are linearly independent and thus Lemma \ref{lem:stab} can easily be transposed. 
%\end{note}

\begin{lem}
The even part of $\psi$ satisfies
\begin{multline*}
\psi_e(x) = \sumlim{j=0}{1} \left( \frac{(-E)^j}{(2j)!} x^{2j} - \frac{Z_0}{2} \frac{(-E)^j}{(2j+1)!} |x|^{2j+1} \right) + W(0) \psi(0) \left( \frac{x^2}{2} - \frac{Z_0}{2} \frac{|x|^3}{3!} \right) \\
+ \left(W'(0) \psi'_s(0) - \frac{E}{2} W(0) \psi(0)\right) \frac{x^4}{4!} - \frac{Z_0}{2} \left(\frac{W''(0)}{2} \psi(0) - E  W(0) \psi(0) \right) \frac{|x|^5}{5!} \\
+ W(0) \psi(0) \left( \frac{x^4}{4!} - \frac{Z_0}{2} \frac{x^5}{5!} \right) + \cO\left(x^6\right),
\end{multline*}
where 
\begin{equation*}
\psi'_s(0) = \frac{\psi'(0_+) + \psi'(0_-)}{2}.
\end{equation*}
\end{lem}

\begin{proof}
The proof follows from Lemma \ref{lem:new_decomp} and a careful estimation of the terms $\iint W \psi$ and $\iiiint W \psi$.
\end{proof}

Since $W$ is not even, $\psi$ does not have the same structure as for the double delta potential. More precisely, we can show that because of the term $\iint W \psi$ the singularity of the fifth order term cannot be removed by the VPAW approach. 

\begin{lem}
For $N=2$ there exist coefficients $c_0$ and $c_1$ such that:
\begin{equation*}
\psi_e(x) - c_0 \phi_0(x) - c_1 \phi_1(x) = \cO\left(x^4\right).
\end{equation*}

For $N \geq 3$, there exists a family of coefficients $(c_k)_{0 \leq k \leq N-1}$ such that:
\begin{equation*}
\psi_e(x) - \sumlim{k=0}{N-1} c_k \phi_k(x) = \cO\left(x^5\right).
\end{equation*}
\end{lem}

Following the same steps as in Section \ref{sec:proof}, we can establish the following theorems.

\begin{theo}[Estimates on the Fourier coefficients]
\label{theo:FourierPsi_pot}
Let $N \in \N^*$ and $d \geq N$. Let $\widehat{\tilde{\psi}}_m$ be the $m$-th Fourier coefficient of $\tilde{\psi}$. Then, there exists a positive constant $C$ such that for all $0 < \eta \leq \eta_0$ and $m \geq \tfrac{1}{\eta}$
\begin{equation*}
\left| \widehat{\tilde{\psi}}_m \right| \leq C \left( \frac{\eta^{2N \wedge 5}}{m^2} + \frac{1}{\eta^{d-1}m^{d+1}} \right),
\end{equation*}
where $a \wedge b = \min (a,b)$.
\end{theo}

\begin{theo}[Estimates on the eigenvalues]
\label{theo:energies_pot}
Let $N \in \N^*$ and $d \geq N$. Let $E_M^\eta$ be an eigenvalue of the variational approximation of \eqref{eq:newH_VPAW} in a basis of $M$ plane-waves and for a cut-off radius $0 < \eta \leq \eta_0$, and let $E$ be the corresponding exact eigenvalue. There exists a constant $C >0$ independent of $\eta$ and $M$ such that for all $0 < \eta \leq \eta_0$ and $M \geq \tfrac{1}{\eta}$
\begin{equation}
\label{eq:eig_pert}
|E_M^\eta - E| \leq C \left( \frac{\eta^{4N \wedge 10}}{M} + \frac{1}{\eta^{2d-2}} \frac{1}{M^{2d-1}} \right).
\end{equation}
\end{theo}

\begin{note}
The estimate \eqref{eq:eig_pert} does not seem optimal as shown in Figure \ref{fig:continu}. It seems that singularities of any order can be removed by the VPAW method. 
\end{note}

\section{Numerical tests}
\label{sec:numerics}

The goal of this section is to compare the theoretical estimates determined in Sections \ref{sec:results} and \ref{sec:proof} to numerical simulations and show that they are optimal. 

All numerical simulations are carried out with $Z_0 = Z_a =10$ and $a=0.4$. The Fourier coefficients are evaluated by a very accurate numerical integration. 

It is interesting to compare the results (Figure \ref{grph:PAW}) obtained by a direct expansion of the wave function $\psi$ (here displayed by the points $N=0$) and the VPAW method. Recall that $N$ is the number of pseudo wave functions used to build the operator $T$. The smoothness of the pseudo wave functions is set to $d=N$. 

\begin{figure}[H]
\centering
	\begin{subfigure}[b]{0.45\textwidth}
	\includegraphics[width = \textwidth]{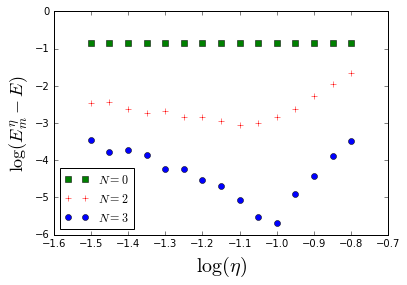}
	\caption{Error on the $8$-th eigenvalue for $M =64$ basis functions}
	\end{subfigure}
	\quad
	\begin{subfigure}[b]{0.45\textwidth}
	\includegraphics[width = \textwidth]{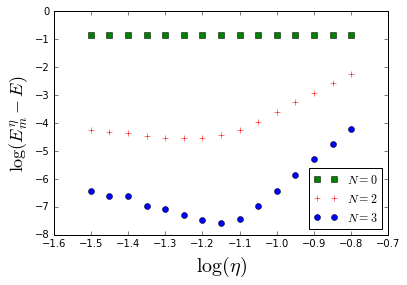}	
	\caption{Error on the $8$-th eigenvalue for $M=256$ basis functions}
	\end{subfigure}
	
\caption{The VPAW method compared to a direct calculation for the 8-th eigenvalue}
\label{grph:PAW}
\end{figure}

Given a number $M$ of basis functions, the VPAW method is much more accurate than the direct method although it is quite sensitive to the choice of $\eta$. More comments on this behavior will be made in Section \ref{sec:eta-dep}. We do not report the computing times for the VPAW method because in this study, each time a simulation is run, we generate all the pseudo wave functions $\tilde{\phi}$, the projector functions $\tilde{p}$ and compute their Fourier coefficients. In practice, these data are precomputed and stored in a file. Thus, the only additional cost compared to the direct method comes from the assembly of the matrices $(\mathrm{Id}+T)^* H (\mathrm{Id}+T)$ and $(\mathrm{Id}+T)^* (\mathrm{Id}+T)$.

\subsection{Derivative jumps}

Since $\psi$ and the functions $\phi_i$ are known analytically, it is possible to evaluate the derivative jumps of $\tilde{\psi}$ at $0$ and $\pm \eta$ (Figures \ref{grph:der_jump} and \ref{tab:der_jump}). The plots are given for the eigenfunction associated to the lowest eigenvalue of $H$. The behavior is similar for other eigenfunctions. 

\begin{figure}[H]
\centering
	\begin{subfigure}[b]{0.45\textwidth}
	\includegraphics[width = \textwidth]{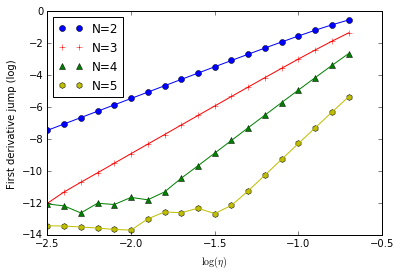}
	\caption{First derivative jump at $0$ as a function of $\eta$ in log-log scale \linebreak ($d=N$).}
	\end{subfigure}
	\quad
	\begin{subfigure}[b]{0.45\textwidth}
	\includegraphics[width = \textwidth]{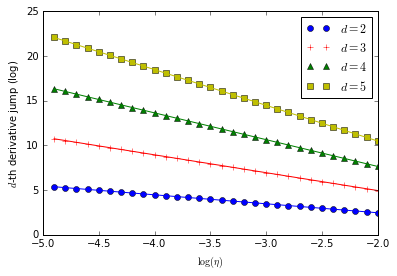}	
	\caption{$d$-th derivative jump at $\pm \eta$ as a function of $\eta$ in log-log scale \linebreak ($N=2$).}
	\end{subfigure}
	
\caption{Derivative jumps of the pseudo wave function $\tilde{\psi}$}
\label{grph:der_jump}
\end{figure}

\begin{figure}[H]
	\begin{subfigure}[b]{0.45\textwidth}
	\centering
	\begin{tabular}{|c|c|c|}
	\hline
	$N$ & Numerics & Theory \\
	\hline
	$2$ & 3.90 & 4 \\
	\hline
	$3$ & 5.94 & 6 \\
	\hline
	$4$ & 7.85 & 8 \\
	\hline
	$5$ & 9.85 & 10 \\
	\hline
	\end{tabular}
	\caption{Numerical and theoretical slopes for the first derivative jump at 0.}
	\end{subfigure}
	\quad
	\begin{subfigure}[b]{0.45\textwidth}
	\centering
	\begin{tabular}{|c|c|c|}
	\hline
	$d$ & Numerics & Theory \\
	\hline
	$2$ & -1.005 & -1 \\
	\hline
	$3$ & -2.000 & -2 \\
	\hline
	$4$ & -3.000 & -3 \\
	\hline	
	$5$ & -4.000 & -4 \\
	\hline
	\end{tabular}
	\caption{Numerical and theoretical slopes for \linebreak the $d$-th~derivative jump at $\pm \eta$.}
	\end{subfigure}
\caption{Comparison of the theoretical and numerical results for the derivative jumps}
\label{tab:der_jump}
\end{figure}

These numerical results are in remarkable agreement with Propositions \ref{lem:der_jump_0} and \ref{lem:der_jump_eta}.

\subsection{Comparison of the PAW and VPAW methods in pre-asymptotic regime}

The simulations are run for a fixed value of $d =6$ and two different values of $\eta$ ($\eta= 0.1$ and $\eta = 0.2$). In Figure \ref{grph:PAW-VPAW}, $E$ is the lowest eigenvalue of the 1D-Schr\"odinger operator $H$ given by \eqref{eq:H_mol}. 

Recall that our theoretical estimate on the eigenvalue given by the VPAW method is :
\begin{equation}
\label{eq:rappel}
|E_M^\eta - E| \leq C \left( \frac{\eta^{4N}}{M} + \frac{1}{\eta^{2d-2}} \frac{1}{M^{2d-1}} \right).
\end{equation}

To transpose the PAW method to our one-dimensional setting, we need to account for the use of a
pseudo-potential. For this purpose, we replace the Dirac delta potential by some smooth function in Equation \eqref{eq:H_mol}. We choose the 1-periodic function $\chi_\epsilon$ such that
\begin{equation*}
\chi_\epsilon (x) = \begin{cases} \frac{C}{\epsilon} \exp \left(-\frac{1}{1-\left(\frac{x}{\epsilon}\right)^2} \right), &  x \in [-\epsilon,\epsilon], \\
0 , &  x \in [-1/2,1/2] \setminus [-\epsilon, \epsilon]. \end{cases}
\end{equation*}
where $C$ ensures that $\int_{-\epsilon}^\epsilon \chi_\epsilon = 1$. As $\epsilon$ goes to 0, $\chi_\epsilon$
converges to the 1-periodized Dirac potential in $H^{-1}_\mathrm{per}(0,1)$. 

As expected, the PAW method quickly converges to a wrong value of $E$. It is interesting to notice that asymptotically, the VPAW convergence is of order $\mathcal{O}\left( \frac{1}{M} \right)$  but for small enough values of $M$ and $\eta$, the second term in the RHS of \eqref{eq:rappel} dominates.

\begin{figure}[H]
\centering
	\begin{subfigure}[b]{0.9\textwidth}
	\includegraphics[width=\textwidth]{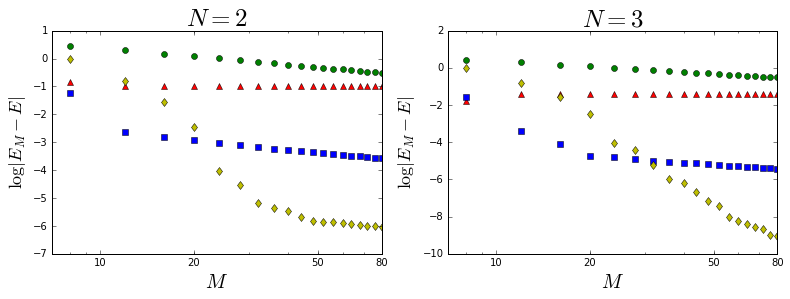}
	\end{subfigure}
	
	\begin{subfigure}[b]{0.7\textwidth}
	\includegraphics[width=\textwidth]{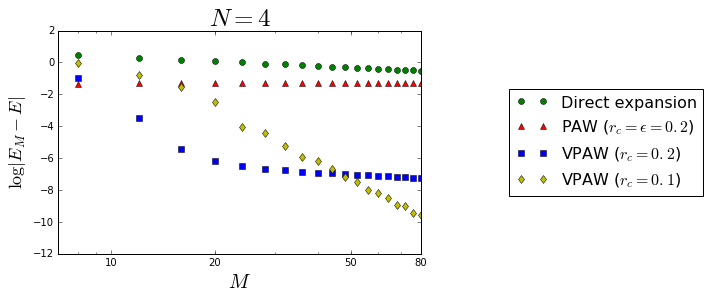}
	\end{subfigure}
\caption{Comparison between the PAW and VPAW methods for the lowest eigenvalue}
\label{grph:PAW-VPAW}
\end{figure}

\subsection{Asymptotic regime}
\label{sec:eta-dep}

\subsubsection{Behavior in the plane-wave cut-off $M$}

The next numerical tests (Figures \ref{grph:M} and \ref{tab:conv_M}) are run with $d=N$ and $N=2$, $N=3$. The pseudo wave function $\tilde{\psi}$ is expanded in $M=2^m$ plane-waves, $m=7$ to $9$. 

\begin{figure}[!h]
\centering
	\begin{subfigure}[b]{0.45\textwidth}
	\includegraphics[width = \textwidth]{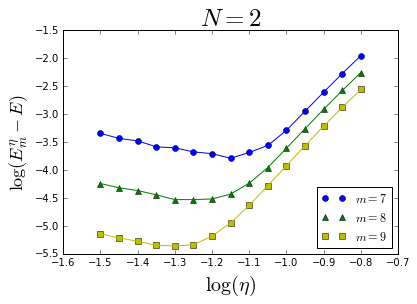}
	\caption{Error on the $9$-th eigenvalue with $N=2$}
	\end{subfigure}
	\quad
	\begin{subfigure}[b]{0.45\textwidth}
	\includegraphics[width = \textwidth]{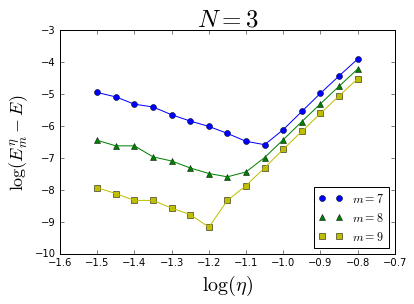}	
	\caption{Error on the $9$-th eigenvalue with $N=3$}
	\end{subfigure}
	
\caption{Error on the eigenvalue for different values of $M$}
\label{grph:M}
\end{figure}

Here, we can clearly see two regimes : for $\eta$ small (resp. $\eta$ large), the leading term in the error is dominated by the $d$-th derivative jumps at $k \pm \eta$ and $k + a \pm \eta$, $k \in \Z$ (resp. the first derivative jump at $k$ and $k + a$, $k \in \Z$). In each regime, the gaps between the decreasing and increasing slopes seem constant and their evaluation gives the correct orders of convergence in $M$ (see Figure \ref{tab:conv_M}).

\begin{figure}[!h]
	\begin{subfigure}[b]{0.45\textwidth}
	\centering
	\begin{tabular}{|c|c|c|}
	\hline
	 & Numerics & Theory \\
	\hline
	{\small Decreasing lines} & 0.30 & $\log(2) \simeq 0.30 $ \\ 
	\hline
	{\small Increasing lines} & 0.90 & $3 \log(2) \simeq 0.90 $ \\ 
	\hline
	\end{tabular}
	\caption{Gaps for $N=2$}
	\end{subfigure}
	\quad
	\begin{subfigure}[b]{0.45\textwidth}
	\centering
	\begin{tabular}{|c|c|c|}
	\hline
		 & Numerics & Theory \\
	\hline
	{\small Decreasing lines} & 0.32 & $\log(2) \simeq 0.30 $ \\ 
	\hline
	{\small Increasing lines} & 1.50 & $5 \log(2) \simeq 1.50 $ \\ 
	\hline
	\end{tabular}
	\caption{Gaps for $N=3$}
	\end{subfigure}
\caption{Estimation of the order of convergence in $M$}
\label{tab:conv_M}
\end{figure}

\subsubsection{Dependence of the convergence rate in $\eta$ on $N$ and $d$}

In each graph of Figure \ref{grph:d}, we have kept $M$ constant to track the dependence of the convergence rate in $\eta$. By Theorem \ref{theo:energies}, the logarithm of error on the eigenvalue is given by
$$
\log (E_M^\eta - E) = \log(C) + \log \left( \frac{\eta^{4N}}{M} + \frac{1}{\eta^{2d-2}M^{2d-1}} \right).
$$
Hence, when $\eta$ is large, we have
$$
\log (E_M^\eta - E) \simeq \log(C) + 4N \log \eta - \log (M),
$$
and when $\eta$ is small, we have
$$
\log (E_M^\eta - E) \simeq \log(C) - (2d -2) \log \eta - (2d-1) \log M.
$$

\begin{figure}[!h]
\centering
	\begin{subfigure}[b]{0.45\textwidth}
	\includegraphics[width = \textwidth]{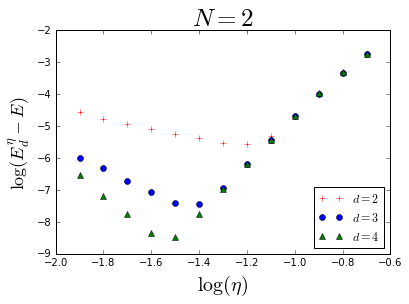}
	\caption{Error on the $8$-th eigenvalue with $N=2$, $M=256$}
	\end{subfigure}
	\quad
	\begin{subfigure}[b]{0.45\textwidth}
	\includegraphics[width = \textwidth]{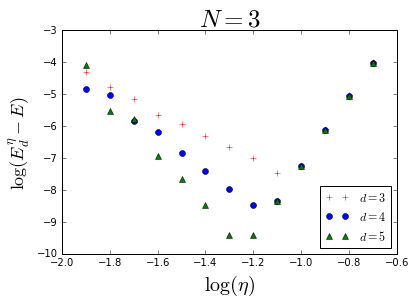}	
	\caption{Error on the $8$-th eigenvalue with $N=3$, $M=128$}
	\end{subfigure}
	
\caption{Error on the eigenvalue for different values of $d$}
\label{grph:d}
\end{figure}

Notice that in each graph, for $\eta$ large, the parameter $d$ has a negligible effect on the error on the eigenvalues, in agreement with our theoretical estimates.

\begin{figure}[!!h]
	\begin{subfigure}[b]{0.45\textwidth}
	\centering
	\begin{tabular}{|c|c|c|}
	\hline
	$d$ & Numerics & Theory \\
	\hline
	$2$ & 6.5 & 8 \\
	\hline
	$3$ & 6.9 & 8 \\
	\hline
	$4$ & 7.2 & 8 \\
	\hline
	\end{tabular}
	\caption{$N=2$}
	\end{subfigure}
	\quad
	\begin{subfigure}[b]{0.45\textwidth}
	\centering
	\begin{tabular}{|c|c|c|}
	\hline
	$d$ & Numerics & Theory \\
	\hline
	$3$ & 10.6 & 12 \\
	\hline
	$4$ & 10.7 & 12 \\
	\hline	
	$5$ & 10.9 & 12 \\
	\hline
	\end{tabular}
	\caption{$N=3$}
	\end{subfigure}
\caption{Estimation of the increasing slopes in Figure \ref{grph:d}}
\end{figure}

There is a small discrepancy between the theoretical and numerical values of the increasing slope. A possible explanation could be that the estimates we have given for the first derivative jumps are valid asymptotically as $\eta$ goes to 0, but the increasing slopes are observed for relatively large values of $\eta$. 

\begin{figure}[!!h]
	\begin{subfigure}[b]{0.45\textwidth}
	\centering
	\begin{tabular}{|c|c|c|}
	\hline
	$d$ & Numerics & Theory \\
	\hline
	$2$ & -1.6 & -2 \\
	\hline
	$3$ & -3.6 & -4 \\
	\hline
	$4$ & -6.0 & -6 \\
	\hline
	\end{tabular}
	\caption{$N=2$}
	\end{subfigure}
	\quad
	\begin{subfigure}[b]{0.45\textwidth}
	\centering
	\begin{tabular}{|c|c|c|}
	\hline
	$d$ & Numerics & Theory \\
	\hline
	$3$ & -3.8 & -4 \\
	\hline
	$4$ & -5.4 & -6 \\
	\hline	
	$5$ & -7.8 & -8 \\
	\hline
	\end{tabular}
	\caption{$N=3$}
	\end{subfigure}
\caption{Estimation of the decreasing slopes in Figure \ref{grph:d}}
\end{figure}

For the decreasing slopes, our estimate is in very good agreement with the numerical simulations. 

\subsection{Perturbation by a continuous potential}
\label{subsec:perturbation}

In this subsection, we study the VPAW method applied to the Hamiltonian \eqref{eq:new_H} with \linebreak $W(x) = 10\sin(2\pi x + 0.2)$. Since this model is not exactly solvable, we use a P2 finite elements method to compute very accurately the eigenvalues (the relative error on the computed eigenvalue is less than $10^{-10}$). 

\begin{figure}[!h]
\centering
\includegraphics[width = 0.6\textwidth]{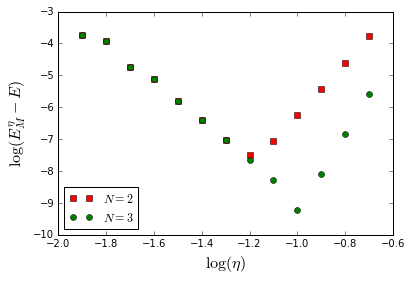}
\caption{Error on the first eigenvalue as a function of $\eta$ ($M=128, d=4$)}
\label{fig:continu}
\end{figure}

\begin{figure}[!!h]
	\begin{subfigure}[b]{0.45\textwidth}
	\centering
	\begin{tabular}{|c|c|c|}
	\hline
	$N$ & Numerics & Theory \\
	\hline
	$2$ & 8.2 & 8 \\
	\hline
	$3$ & 12 & 10 \\
	\hline
	\end{tabular}
	\caption{Increasing slopes}
	\end{subfigure}
	\quad
	\begin{subfigure}[b]{0.45\textwidth}
	\centering
	\begin{tabular}{|c|c|c|}
	\hline
	 & Numerics & Theory \\
	\hline
	$N=2$ and $N=3$ & -5.7 & -6 \\
	\hline	
	\end{tabular}
	\caption{Decreasing slopes}
	\end{subfigure}
\caption{Estimation of the slopes in Figure \ref{fig:continu}}
\end{figure}

For $N=3$, the increasing part of the curve has a slope which is very close to the theoretical estimation of Theorem \ref{theo:energies} (that is with $W=0$). This seems to indicate that the VPAW method removes the singularity at the nucleus up to the fifth order, but we are unable to support this observation with rigorous numerical analysis arguments.

%\bibliographystyle{siam}
%\bibliography{PAW-articles}

\begin{thebibliography}{10}

\bibitem{audouze2006projector}
{\sc C.~Audouze, F.~Jollet, M.~Torrent, and X.~Gonze}, {\em Projector
  augmented-wave approach to density-functional perturbation theory}, Physical
  Review B, 73 (2006), p.~235101.

\bibitem{blanc2017paw}
{\sc X.~Blanc, E.~Canc\`es, and M.-S. Dupuy}, in preparation.

\bibitem{blanc2017}
\leavevmode\vrule height 2pt depth -1.6pt width 23pt, {\em Variational
  projector augmented-wave method}, Comptes Rendus Mathematique, 355 (2017),
  pp.~665 -- 670.

\bibitem{blochl94}
{\sc P.~E. Blochl}, {\em Projector augmented-wave method}, Phys. Rev. B, 50
  (1994), pp.~17953--17979.

\bibitem{cances2017discretization}
{\sc E.~Canc{\`e}s and G.~Dusson}, {\em Discretization error cancellation in
  electronic structure calculation: a quantitative study}, accepted at ESAIM:
  Mathematical Modelling and Numerical Analysis,  (2017).

\bibitem{cances2016existence}
{\sc E.~Cances and N.~Mourad}, {\em Existence of a type of optimal
  norm-conserving pseudopotentials for kohn--sham models}, Communications in
  Mathematical Sciences, 14 (2016), pp.~1315--1352.

\bibitem{chen2015numerical}
{\sc H.~Chen and R.~Schneider}, {\em Numerical analysis of augmented plane wave
  methods for full-potential electronic structure calculations}, ESAIM:
  Mathematical Modelling and Numerical Analysis, 49 (2015), pp.~755--785.

\bibitem{dolg2011relativistic}
{\sc M.~Dolg and X.~Cao}, {\em Relativistic pseudopotentials: their development
  and scope of applications}, Chemical reviews, 112 (2011), pp.~403--480.

\bibitem{fournais2005sharp}
{\sc S.~Fournais, M.~Hoffmann-Ostenhof, T.~Hoffmann-Ostenhof, and T.~{\O}.
  S{\o}rensen}, {\em Sharp regularity results for coulombic many-electron wave
  functions}, Communications in mathematical physics, 255 (2005), pp.~183--227.

\bibitem{goedecker1996separable}
{\sc S.~Goedecker, M.~Teter, and J.~Hutter}, {\em Separable dual-space gaussian
  pseudopotentials}, Physical Review B, 54 (1996), p.~1703.

\bibitem{hoffmann2001electron}
{\sc M.~Hoffmann-Ostenhof, T.~Hoffmann-Ostenhof, and T.~{\O}. S{\o}rensen},
  {\em Electron wavefunctions and densities for atoms}, in Annales Henri
  Poincar{\'e}, vol.~2, Springer, 2001, pp.~77--100.

\bibitem{jollet2014generation}
{\sc F.~Jollet, M.~Torrent, and N.~Holzwarth}, {\em Generation of projector
  augmented-wave atomic data: A 71 element validated table in the {XML}
  format}, Computer Physics Communications, 185 (2014), pp.~1246--1254.

\bibitem{kato1957eigenfunctions}
{\sc T.~Kato}, {\em On the eigenfunctions of many-particle systems in quantum
  mechanics}, Communications on Pure and Applied Mathematics, 10 (1957),
  pp.~151--177.

\bibitem{kleinman1982efficacious}
{\sc L.~Kleinman and D.~Bylander}, {\em Efficacious form for model
  pseudopotentials}, Physical Review Letters, 48 (1982), p.~1425.

\bibitem{koelling1975use}
{\sc D.~Koelling and G.~Arbman}, {\em Use of energy derivative of the radial
  solution in an augmented plane wave method: application to copper}, Journal
  of Physics F: Metal Physics, 5 (1975), p.~2041.

\bibitem{kresse1996efficient}
{\sc G.~Kresse and J.~Furthm\"uller}, {\em Efficient iterative schemes for
  \textit{ab initio} total-energy calculations using a plane-wave basis set},
  Phys. Rev. B, 54 (1996), pp.~11169--11186.

\bibitem{kresse99}
{\sc G.~Kresse and D.~Joubert}, {\em From ultrasoft pseudopotentials to the
  projector augmented-wave method}, Phys. Rev. B, 59 (1999), pp.~1758--1775.

\bibitem{laasonen93}
{\sc K.~Laasonen, A.~Pasquarello, R.~Car, C.~Lee, and D.~Vanderbilt}, {\em
  Car-{P}arrinello molecular dynamics with {V}anderbilt ultrasoft
  pseudopotentials}, Phys. Rev. B, 47 (1993), pp.~10142--10153.

\bibitem{levitt2015parallel}
{\sc A.~Levitt and M.~Torrent}, {\em Parallel eigensolvers in plane-wave
  density functional theory}, Computer Physics Communications, 187 (2015),
  pp.~98--105.

\bibitem{rostgaard2009projector}
{\sc C.~Rostgaard}, {\em The projector augmented-wave method}, arXiv preprint
  arXiv:0910.1921,  (2009).

\bibitem{singh2006planewaves}
{\sc D.~J. Singh and L.~Nordstrom}, {\em Planewaves, Pseudopotentials, and the
  LAPW method}, Springer Science \& Business Media, 2006.

\bibitem{torrent2008337}
{\sc M.~Torrent, F.~Jollet, F.~Bottin, G.~Zérah;, and X.~Gonze}, {\em
  Implementation of the projector augmented-wave method in the {ABINIT} code:
  Application to the study of iron under pressure}, Computational Materials
  Science, 42 (2008), pp.~337 -- 351.

\bibitem{troullier1991efficient}
{\sc N.~Troullier and J.~L. Martins}, {\em Efficient pseudopotentials for
  plane-wave calculations}, Physical review B, 43 (1991), p.~1993.

\bibitem{vanderbilt90}
{\sc D.~Vanderbilt}, {\em Soft self-consistent pseudopotentials in a
  generalized eigenvalue formalism}, Phys. Rev. B, 41 (1990), pp.~7892--7895.

\bibitem{weinberger1974variational}
{\sc H.~F. Weinberger}, {\em Variational methods for eigenvalue approximation},
  vol.~15, Siam, 1974.

\end{thebibliography}

\end{document}